\documentclass[final]{siamart190516} 

\usepackage{amsfonts,amsopn,braket,amsmath,amssymb}
\usepackage{mathrsfs}
\usepackage{color}

\title{A model reduction approach for inverse problems with operator valued data
}
\author{J\"urgen D\"olz%
\thanks{Institute for Numerical Simulation, University of Bonn, Friedrich-Hirzebruch-Allee 7, 53115 Bonn, Germany.
\email{doelz@ins.uni-bonn.de}}
\and Herbert Egger%
\thanks{Numerical Analysis and Scientific Computing, Department of Mathematics, TU Darmstadt, Dolivostr. 15, 64293 Darmstadt, Germany. 
\email{egger@mathematik.tu-darmstadt.de}}
\and Matthias Schlottbom% 
\thanks{Department of Applied Mathematics, University of Twente,
P.O. Box 217, 7500 AE Enschede, The Netherlands.
\email{m.schlottbom@utwente.nl}}
}

\headers{Inverse problems with operator valued data}{J. D\"olz, H. Egger, and M. Schlottbom}

\def\NN{\mathbb{N}}
\def\RR{\mathbb{R}}
\def\eps{\epsilon}
\def\dim{\operatorname{dim}}
\def\rank{\operatorname{rank}}

\def\mycirc{\,}

\def\HS{\mathbb{HS}}

\def\XX{\mathbb{X}}
\def\YY{\mathbb{Y}}
\def\ZZ{\mathbb{Z}}
\def\VV{\mathbb{V}}
\def\UU{\mathbb{U}}
\def\AA{\mathbb{A}}
\def\BB{\mathbb{B}}
\def\CC{\mathbb{C}}

\def\T{\mathcal{T}}
\def\U{\mathcal{U}}
\def\V{\mathcal{V}}
\def\D{\mathcal{D}}
\def\M{\mathcal{M}}

\def\L{\mathcal{L}}
\def\P{\mathcal{P}}
\def\Q{\mathcal{Q}}
\def\R{\mathcal{R}}
\def\S{\mathcal{S}}

\def\Id{\mathcal{I}}
\def\Th{\Omega_h}

\def\tta{\mathtt{a}}
\def\ttc{\mathtt{c}}
\def\ttd{\mathtt{d}}
\def\ttx{\mathtt{x}}
\def\ttm{\mathtt{m}}

\def\fro{\mathbb{F}}
\def\ttA{\mathtt{A}}
\def\ttB{\mathtt{B}}
\def\ttD{\mathtt{D}}
\def\ttE{\mathtt{E}}
\def\ttK{\mathtt{K}}
\def\ttI{\mathtt{I}}
\def\ttM{\mathtt{M}}
\def\ttN{\mathtt{N}}

\def\ttQ{\mathtt{Q}}
\def\ttR{\mathtt{R}}
\def\ttS{\mathtt{S}}
\def\ttU{\mathtt{U}}
\def\ttT{\mathtt{T}}
\def\ttTt{\mathtt{Tt}}
\def\ttt{\mathtt{t}}
\def\ttKt{\mathtt{Kt}}
\def\ttV{\mathtt{V}}
\def\ttA{\mathtt{A}}
\def\ttX{\mathtt{X}}
\def\ttY{\mathtt{Y}}
\def\ttZ{\mathtt{Z}}

\newsiamremark{assumption}{Assumption} 
\crefname{assumption}{Assumption}{Assumption} % <- Preamble
\newsiamremark{remark}{Remark} 
\crefname{remark}{Remark}{Remark} % <- Preamble

\begin{document}

\maketitle

\begin{abstract} 
We study the efficient numerical solution of linear inverse problems with operator valued data which arise, e.g., in seismic exploration, inverse scattering, or tomographic imaging. The high-dimensionality of the data space implies extremely high computational cost already for the evaluation of the forward operator which makes a numerical solution of the inverse problem, e.g., by iterative regularization methods, practically infeasible. To overcome this obstacle, we take advantage of the underlying tensor product structure of the problem and propose a strategy for constructing low-dimensional certified reduced order models of quasi-optimal rank for the forward operator which can be computed much more efficiently than the truncated singular value decomposition. A complete analysis of the proposed model reduction approach is given in a functional analytic setting and the efficient numerical construction of the reduced order models as well as of their application for the numerical solution of the inverse problem is discussed. In summary, the setup of a low-rank approximation can be achieved in an offline stage at essentially the same cost as a single evaluation of the forward operator, while the actual solution of the inverse problem in the online phase can be done with extremely high efficiency. The theoretical results are illustrated by application to a typical model problem in fluorescence optical tomography.
\end{abstract}



% \bigskip

\section{Introduction} \label{sec:1}

We consider the numerical solution of linear inverse problems with operator valued data modeled by abstract operator equations
\begin{align} \label{eq:ip}
\T(c) = \M^\delta. 
\end{align}
Here $c \in \XX$ is the quantity to be determined and we assume that $\M^\delta : \YY \to \ZZ'$, representing the possibly perturbed measurements, is a linear operator of Hilbert-Schmidt class between Hilbert spaces $\YY$ and $\ZZ'$, the dual of $\ZZ$. 
We further assume that the forward operator $\T : \XX \to \HS(\YY,\ZZ')$ is linear and compact, and admits a factorization of the form  
\begin{align} \label{eq:def_T}
\T(c) = \V' \, \D(c) \, \U,
\end{align}
with $\V'$, $\D(c)$, and $\U$ again denoting appropriate linear operators. 
Problems of this kind arise in a variety of applications, e.g. in fluorescence tomography \cite{ArridgeSchotland2009,Ntziachristos2006}, inverse scattering \cite{ColtonKress,Grinberg2008}, or source identification \cite{HohageEtAl2020}, but also as linearizations of related nonlinear inverse problems, see e.g., \cite{EggFreiSch10,Somersalo1992} or \cite{Levinson2016} and the references given there. 
In such applications,  $\U$ typically models the propagation of excitation fields generated by the sources, $\D$ describes the interaction with the medium to be probed, and $\V'$ models the emitted fields which can be recorded by the detectors.
In the following, we briefly outline our basic approach towards the numerical solution of \cref{eq:ip}--\cref{eq:def_T} and report about related work in the literature.

\subsection{Regularized inversion}\label{sec:reginv}

By the particular functional analytic setting, 
the inverse problem \cref{eq:ip}--\cref{eq:def_T} amounts to an ill-posed linear operator equation in Hilbert spaces and standard regularization theory can be applied for its stable solution \cite{Bakushinsky2004,EHN96}. 
Following standard arguments, we assume that $\M^\delta$ is a perturbed version of the exact data $\M$ and that a bound on the measurement noise
\begin{align} \label{eq:noise}
\|\M - \M^\delta\|_{\HS(\YY,\ZZ')} \le \delta
\end{align}
is available. We further denote by $c^\dag$ the minimum norm solution of \cref{eq:ip} with $\M^\delta$ replaced by $\M = {\T(c^\dag)}$. 
A stable approximation for the solution $c^\dag$ can then be obtained by the regularized inversion of \cref{eq:ip}, e.g., by spectral regularization methods
\begin{align} \label{eq:cad} 
c_\alpha^\delta 
&=  g_\alpha(\T^\star \T) \T^\star \M^\delta 
 = \T^\star g_\alpha(\T \T^\star) \M^\delta.
\end{align}
Here $\T^\star : \HS(\YY,\ZZ') \to \XX$ denotes the adjoint of the operator $\T$ 
{and $g_\alpha(\lambda)$ denotes a regularized approximation of $1/\lambda$, i.e., the filter function, satisfying some standard conditions; we refer to \cite[Chapter 2]{Bakushinsky2004} or \cite[Chapter 4]{EHN96} for details and to \cite{Mathe2003} for generalizations.}
A typical example for the filter function is $g_\alpha(\lambda) = (\lambda+\alpha)^{-1}$, which leads to Tikhonov regularization $c_\alpha^\delta = (\T^\star \T + \alpha \Id)^{-1} \T^\star \M^\delta$, see {\cite{TikhonovArsenin}. 
Another filter function with certain optimality conditions results from truncated singular value decomposition and reads $g_\alpha(\lambda)=1/\lambda$ if $\lambda\geq\alpha$ and $g_\alpha(\lambda)=0$ otherwise.} 

{Note that the choice of the regularization norm in \eqref{eq:cad} is incorporated implicitly in the definition of the function spaces and one can obtain convergence $c_\alpha^\delta \to c^\dagger$ of the regularized solutions to the minimum norm solution $c^\dagger$ in this norm if $\delta \to 0$ and $\alpha=\alpha(\delta,\M^\delta)$ is chosen appropriately; the rate of convergence will depend on the properties of the filter function $g_\alpha$, the parameter choice $\alpha(\delta,\M^\delta)$, and on the smoothness of the minimum norm solution $c^\dagger$; see  \cite{Bakushinsky2004,EHN96,Mathe2003} for details.  
For tomographic applications we have in mind, uniqueness results are usually available \cite{Isakov2017,Natterer2001}, i.e., $\T$ can be assumed to be injective, in which case $c^\dagger$ is independent of the regularization norm; see Remark~\ref{rem:unique} below.
We will not step further into the analysis of regularization methods, but rather focus on their efficient numerical realization for problems with operator valued data.}

For the actual computation of the regularized solution $c_\alpha^\delta$, a sufficiently accurate finite dimensional approximation of the operator $\T$ is required, which is usually obtained by some discretization procedure; in the language of model order reduction, this is called the \emph{truth} or \emph{high-fidelity approximation} \cite{Benner2015,Quarteroni2016}. 
In the following discussion, we will not distinguish between infinite dimensional operators and their truth approximations. 
We thus assume that $\dim(\XX) = m$, $\dim(\YY) = k_\YY$ and $\dim(\ZZ)=k_\ZZ$. For ease of notation, we assume that $k_\XX=k_\YY=k$ in the following.
We may then identify $c$ with a vector in $\RR^m$, $\M^\delta$ with a matrix in $\RR^{k \times k}$, and $\T$ with a third order tensor in $\RR^{k \times k \times m}$ or a matrix in $\RR^{k^2 \times m}$. 
In typical applications, like computerized tomography, the dimensions $m$ and $k$ are very large and one may assume that $k < m < k^2$; see \cite{Chaillat2012,Markel2019}. The inverse problem \eqref{eq:ip} thus can be considered to be typically of \emph{large scale} and \emph{overdetermined}. 

\subsection{Model reduction and computational complexity}\label{sec:mor_general}

The high dimensionality of the problem poses severe challenges for the numerical solution of the inverse problem \cref{eq:ip}--\cref{eq:def_T} and different model reduction approaches have been proposed to reduce the computational complexity. 
We consider approximations 
\begin{align} \label{eq:TN}
\T_N = \Q_N \T, 
\end{align}
based on projection in data space, where $\Q_N$ is an orthogonal projection with finite rank $N$, which is the dimension of the range of $\Q_N$. 
Since $\T$ is assumed compact, we can always choose $N$ sufficiently large such that
\begin{align} \label{eq:pert}
\|\T_N - \T\|_{\XX \to \HS(\YY,\ZZ')} \le \delta,
\end{align}
and we may assume that typically $N \ll m,k$, where $m$, $k$ are the dimensions of the truth approximation used for the computations. 
%

% Let us recall that the approximation of minimal rank $N$, satisfying \cref{eq:pert}, is obtained by truncated singular value decomposition of the operator $\T$, which will serve as the benchmark in the following discussion. 

For the stable and efficient numerical solution of the inverse problem \cref{eq:ip}--\cref{eq:def_T}, we may then consider the low-dimensional regularized approximation
\begin{align} \label{eq:cand}
c_{\alpha,N}^\delta = \T_N^\star g_\alpha(\T_N \T_N^\star) \Q_N \M^\delta.
\end{align}
As shown in \cite{Neubauer1988}, the low-rank approximation $c_{\alpha,N}^\delta$ defined in \cref{eq:cand} has essentially the same quality as the infinite dimensional approximation $c_\alpha^\delta$, as long as the perturbation bound \cref{eq:pert} can be guaranteed. 
In the sequel, we therefore focus on the numerical realization of \cref{eq:cand}, which can be roughly divided into the following two stages: \\[-0.5em]
\begin{itemize}\itemsep0.5em
 \item Setup of the approximations $\Q_N$, $\T_N^\star$, and $\T_N \T_N^\star$.  This compute intensive part can be done in an \emph{offline stage} and the constructed approximations can be used for repeated  solution of the inverse problem  \cref{eq:ip} for multiple data.
 \item Computation of the regularized solution \cref{eq:cand}. This \emph{online stage}, which is relevant for the actual solution of \cref{eq:ip}, comprises the following three steps:\\
\begin{center}
\setlength{\tabcolsep}{0.5em}
{\renewcommand{\arraystretch}{1.2}
\begin{tabular}{l|l|l|l}
step        & computations                 & complexity & memory \\
\hline
compression & $\M_N^\delta = \Q_N \M^\delta $ & $N k^2$    & $N k^2$ \\
\hline
analysis    & $z_{\alpha,N}^\delta = g_\alpha(\T_N \T_N^\star) \M_N^\delta $ & $N^2$ & $N^2$ \\ 
\hline
synthesis   & $c_{\alpha,N}^\delta = \T_N^\star z_{\alpha,N}^\delta$ & $N m$ & $Nm$  
\end{tabular}
}
\end{center}
\bigskip
\end{itemize}
%
% For the complexity and memory estimates above, we assumed that $\T \in \RR^{k \times k \times k}$ is the truth approximation obtained after discretization. 
%
Let us note that the analysis step is completely independent of the large system dimension $k,m$ of the truth approximation and therefore the compression and synthesis step are the compute intensive parts in the online stage. 
If $k^2 > m > k$, which is the typical situation \cite{Chaillat2012,Markel2019}, the data compression turns out to be the most compute and memory expensive step. As we will explain below, the tensor product structure of \cref{eq:def_T} allows us to considerably reduce the memory cost in the compression step.

\subsection{Low-rank approximations}\label{sec:low_rank}

A particular example of a low-rank approximation \eqref{eq:TN} is given by the truncated singular value decomposition $\T_{N^{\text{svd}}}$, for which $\Q_{N^{\text{svd}}}$ amounts to the orthogonal projection onto the $N^{\text{svd}}$-dimensional subspace of the left singular vectors  corresponding to the largest singular values $\sigma_1 \ge \ldots \ge \sigma_{N^{\text{svd}}}$ of the operator $\T$.
By construction of the singular value decomposition, we have 
\begin{align} \label{eq:svd2}
    \|\T_{N^{\text{svd}}} - \T\|_{\XX\to\HS(\YY,\ZZ')} = \sigma_{N^{\text{svd}}+1},
\end{align}
which allows to guarantee the desired accuracy \eqref{eq:pert} by choosing $\sigma_{N^{\text{svd}}+1} \le \delta < \sigma_{N^{\text{svd}}}$.  
From the Eckhard-Young-Mirsky theorem, we can conclude that $N=N^{\text{svd}}$ is the minimal rank of an operator $\T_N$ satisfying the perturbation bound bound \eqref{eq:pert},
{i.e., the truncated singular value decomposition certainly yields the best possible low-rank approximation with a given rank $N$.}

Based on Fourier techniques, fast analytic singular value decompositions for linear operators arising in optical diffusion tomography have been constructed in \cite{Markel2003} for problems with regular geometries and constant coefficients.
In more general situations, the full assembly and decomposition of the operator $\T$ is, however, computationally prohibitive.
Krylov subspace methods \cite{Hochstenbach2000,Stoll2012} and randomized  algorithms \cite{HMT2011,Musco2015} then provide alternatives that allow to construct approximate singular value decompositions using only a moderate number of evaluations of $\T$ and its adjoint $\T^\star$.
By combining randomized singular value decompositions for subproblems associated to a single frequency in a recursive manner, approximate singular value decompositions for inverse medium problems have been constructed in \cite{Chaillat2012}.

In a recent work \cite{Markel2019}, motivated by \cite{Levinson2016} and \cite{Lev-Ari2005}, finite dimensional inverse scattering problems of the particular form
\begin{align}
    T(c):=V^\top D(c) U = M^\delta,
    \qquad \text{with} \qquad 
%    D(c)=\operatorname{diag}(\operatorname{diag}(D(c))),
    D(c)=\operatorname{diag}(c)
\end{align}
are considered. Using the Kathri-Rao product $(A \odot B^\top)_{ij,l} = A_{i,l} B_{j,l}$ for matrices $A,B \in \RR^{k \times m}$ with $ij=(k-1) i+j$ this problem can be cast into 
a linear system
\begin{align}
    (U^\top \odot V) \,  c = \operatorname{vec}(M^\delta),
\end{align}
where $\text{vec}(M) \in \RR^{k^2}$ denotes the vectorization of the matrix $M$ by columns. 
The Khatri-Rao product structure allows the efficient evaluation of $T^\top T$, required for the solution of the inverse problem, using pre-computed low-rank approximations for $U \, U^\top$ and $V \, V^\top$; we refer to  \cite{Kolda2009} for a definition and properties of the Kathri-Rao product and a survey on tensor decompositions. 
Apart from the more restrictive assumptions on the problem structure, the computational cost of the reconstruction algorithms in \cite{Markel2019} is still rather high, since the dimension $m$ of the parameter $c$ still appears in the system, which may be prohibitive for problems with distributed parameters. 

Another popular strategy towards dimension reduction for inverse problems with multiple excitations consists in synthetically reducing the number of sources. Such \emph{simultaneous} or \emph{encoded sources} have been used, e.g., in geophysics \cite{Herrmann2009,Krebs2009} and tomography \cite{Ascher2012}; see \cite{Roosta2014} for further references. 
The systematic construction of low-rank approximations is investigated intensively also in the context of model order reduction; see \cite{Benner2015,Quarteroni2016} for a survey on results in this direction.

{Let us note that in the context of inverse problems, the mapping $\T$ here has to be understood as a linear operator, i.e., a tensor of order $2$, and the norms in which the approximation quality should be measured are prescribed by the functional-analytic setting; see \eqref{eq:svd2}. By the Eckhard-Young-Mirsky theorem, the truncated singular value decomposition therefore yields the optimal low-rank approximation, and the main aspect here is to compute the truncated singular value decomposition of the operator $\T$, or a sufficiently good approximation, with minimal effort.
As we will see, this can be done on a rather general level, only using the particular form \eqref{eq:def_T} and some abstract smoothness conditions on the involved operators.}

\subsection{Contributions and outline of the paper}

The main scope of this paper is the systematic construction and analysis of low-rank approximations $\T_N$ for operators $\T$ of the particular form \eqref{eq:def_T} with 
\begin{itemize}
    \item \textit{certified} approximation error bounds \eqref{eq:pert}, and 
    \item \textit{quasi-optimal} rank $N$ comparable to that of the truncated singular value decomposition.
\end{itemize}
The stable solution of the inverse problem \eqref{eq:ip} can then be achieved by \eqref{eq:cand} in a highly efficient manner. The tensor-product structure \eqref{eq:def_T} will further allow us to 
\begin{itemize}
    \item set up $\T_N$ at essentially the same cost as a \textit{single} evaluation of $\T(c)$;
    \item compress the data $\M^\delta$ on the fly already during recording.
\end{itemize}
Before diving into the detailed analysis of our approach, let us briefly highlight the main underlying principles and key steps of the construction.

\subsubsection{Sparse tensor product compression}\label{sec:tpcompr}

Let $\Q_{K,\U}$, $\Q_{K,\V}$ denote orthogonal projections of rank $K$ in the space $\YY$ of sources and the space $\ZZ$ of detectors, respectively. 
If $\dim(\YY)=\dim(\ZZ)=k$, then clearly $K \le k$. 
One may also use different ranks $K_\U$, $K_\V$ for the two approximations, but for ease of notation we take $K=K_\U=K_\V$. 
In the language of \cite{Herrmann2009,Krebs2009}, the columns of the operators $\Q_{K,\U}$ and $\Q_{K,\V}$ amount to \emph{optimal sources} and \emph{detectors}, and the appropriate choice of  $\Q_{K,\U}$ and $\Q_{K,\V}$ is also related to \emph{optimal experiment design} \cite{Pukelsheim2006}. 
We define corresponding approximations  
\begin{align*}
    \U_K=\U \Q_{K,\U} 
    \qquad \text{and} \qquad
    \V_K=\V \Q_{K,\V}
\end{align*}
for the operators $\U: \YY \to \UU$ and $\VV : \ZZ \to \VV$, each of rank $K$. Since $\U$ and $\V$ are compact operators, we may choose $K$ large enough such that the error in these approximations is as small as desired. 
The resulting tensor product approximation \begin{align*}
    \T_{K,K}(c)
= \V_K' \D(c) \, \U_K
\end{align*}
may then be used as an approximation for $\T$.
Following our notation, we may write $\T_{K,K}=\Q_{K,K} \T$, with $\Q_{K,K}$ denoting a tensor-product projection in data space. 
Unfortunately, the rank of $\T_{K,K}$ is in general $K^2$, which turns out to be typically much larger than the optimal rank achievable by truncated singular value decomposition of the same accuracy.
Instead of $\T_{K,K}$, we therefore consider a hyperbolic-cross approximation \cite{Dung2018}, which has the general form
\begin{align*}
    \T_{\widehat{K}} = \Q_{\widehat{K}} \T_{K,K} = \Q_{\widehat{K}} \T,
\end{align*}
with an orthogonal projection $\Q_{\widehat K}$ onto the $\widehat K$ most significant components in the range of $\T_{K,K}$. 
We will show in detail how to construct sparse-tensor product approximations $\T_{\widehat{K}}$ of any desired accuracy, using knowledge about $\U_K$, $\V_K$ and $\D(c)$ only. 
Moreover, under some mild conditions on the operators $\U$, $\V$, we will see that $\widehat K \approx K$ is sufficient to guarantee essentially the same accuracy as the full tensor-product approximation.  
Let us further note that $\T_{\widehat K}$ does not have a tensor-product structure, but is formally based on a tensor-product approximation, which turns out to be advantageous when computing the projected data $\M^\delta_{\widehat{K}} = \Q_{\widehat{K}} \M^\delta$; see below.  

\subsubsection{Recompression}

% {\color{purple}
% Hier hab ich auch noch ein wenig am Text geändert:

By truncated singular value decomposition, we can further reduce the rank of the sparse-tensor product approximation $\T_{\widehat{K}}$ leading to a final approximation 
\begin{align} \label{eq:QN}
\T_N = \Q_N \T 
\qquad \text{with} \qquad 
\Q_N = \P_N \Q_{\widehat{K}} = \P_N \Q_{K,K},
\end{align}
% {\color{blue}
% $\Q_N=\P_N \Q_{\widehat{K}}\Q_{K,K}$? Maybe change the notation as $\P_N$ is the T-SVD projector of $\T_{\hat K}$, which is itself linear, but the projection changes with $\T_{\hat K}$. Maybe we use $P^\delta_N$ in line with Lemma 2.7 below to indicate the dependence...}
which can be shown to have essentially the rank $N \approx N^\textrm{svd}$ of the truncated singular value decomposition of $\T$ with the same accuracy; see \Cref{sec:recompression} for details.
We thus obtain computable approximations $\T_N$ for $\T$ with essentially the same rank as the truncated singular value decomposition of similar accuracy.
Only some mild assumptions on the mapping properties of the operators $\U$ and $\V$ are required to rigorously establish and guarantee the approximation property \cref{eq:pert}.
%
%}

\subsubsection{Summary of basic properties}

% {\color{purple}
% Das hier ist neu:
%
It turns out that the proposed two-step construction of the approximation $\T_{N}$, which is based on the underlying tensor-product structure of the problem, has significant advantages compared to the truncated singular value decomposition in the setup, i.e., $\T_N$ can be computed at the computational cost of essentially \emph{one single evaluation of the forward operator $\T(c)$}.  
Moreover, the underlying tensor-product structure also allows to compute the projection 
$\M^\delta_N = \Q_N \M^\delta$  in an efficient manner. By construction of the projection $\Q_N=\P_N \Q_{K,K}$, we have $\M^\delta_N= \P_N \M_{K,K}$ with pre-compressed data 
\begin{align*} % \label{eq:QKKMd}
\M_{K,K}^\delta
= \Q_{K,K} \M^\delta 
= (\Q_{K,\V}' \M^\delta) \Q_{K,\U} 
\end{align*}
that can be computed by two separate projections $\Q_{K,\U}$, $\Q_{K,\V}$ of rank $K$ in the spaces $\YY$ and $\ZZ$ of sources and detectors.
Since the projection $\Q_{K,\V}'$ can already be applied during recording, simultaneous access to the full data $\M^\delta$ is never required. As a consequence, the memory cost of data recording and compression is thereby reduced to $3 K k + K^2$.
%}

%Our analysis reveals that actually only a hyperbolic cross approximation \cite{Dung2018} for the tensor product approximation is required for the recompression, which substantially improves the computational complexity.
%

\subsubsection{Outline}

% Outline and main contributions
The remainder of the manuscript is organized as follows: 
In \Cref{sec:2}, we discuss in detail the construction of quasi-optimal low-rank approximations $\T_N$ for problems of the form \cref{eq:def_T} with guaranteed accuracy \cref{eq:pert}.
To illustrate the applicability of our theoretical results, we discuss in \Cref{sec:3} a particular example stemming from fluorescence diffuse optical tomography.
An appropriate choice of function spaces allows us to verify all conditions required for the analysis of our approach.
In \Cref{sec:4}, we report in detail about numerical tests, in which we demonstrate the computational efficiency of the model reduction approach and the resulting numerical solution of the inverse problems.

\section{Analysis of the model reduction approach} \label{sec:2}

We will start with introducing our basic notation and then provide a complete analysis of the data compression and model reduction approach outlined in the introduction. 

\subsection{Notation}

Function spaces will be denoted by $\AA,\BB,\ldots$ and assumed to be separable Hilbert spaces with scalar product $(\cdot,\cdot)_\AA$ and norm $\|\cdot\|_\AA$. 
By $\AA'$ we denote the dual of $\AA$, i.e., the space of bounded linear functionals on $\AA$, and by $\langle a',a\rangle_{\AA' \times \AA}$ the corresponding duality product.
Furthermore, $\L(\AA,\BB)$ denotes the Banach space of linear operators $\S: \AA \to \BB$ with norm $\|\S\|_{\L(\AA,\BB)} = \sup_{\|a\|_\AA=1} \|\S a\|_\BB  < \infty$.
We write $\R(\S)=\{\S a : a \in \AA\}$ for the range of the operator $\S$ and define $\rank(\S) = \dim(\R(\S))$.
By $\S' : \BB' \to \AA'$ and $\S^\star: \BB \to \AA$ we denote the dual and the adjoint of a bounded linear operator $\S \in \L(\AA,\BB)$ defined, respectively, for all $a \in \AA$, $b \in \BB$, and $b' \in \BB'$ by 
\begin{align} \label{eq:dual}
\langle \S' b', a\rangle_{\AA' \times \AA} = \langle b',\S a\rangle_{\BB' \times \BB} 
\qquad \text{and} \qquad 
(\S^\star b,a)_\BB &= (b,\S a)_\AA.
\end{align}
The two operators $\S'$ and $\S^\star$ are directly related by Riesz-isomorphisms.  

Any operator $\S^\delta:\AA \to \BB$ with $\|\S - \S^\delta\|_{\L(\AA,\BB)} \lesssim \delta$ will be called a \emph{$\delta$-approximation for $\S$} in the following.
Let us recall that any compact linear operator $\S : \AA \to \BB$ has a singular value decomposition, i.e., a countable system $\{(\sigma_{k},a_{k},b_k)\}_{k \ge 1}$ such that 
\begin{align} \label{eq:svd}
\S a = \sum\nolimits_{k \ge 1} (a,a_k)_{\AA} \sigma_k  b_k,
\end{align}
with singular values $\sigma_1 \ge \sigma_2 \ge \ldots \geq 0$ and $\{a_k: \sigma_k>0\}$ and $\{b_k:\sigma_k>0\}$ denoting orthonormal basis for $\R(\S^\star) \subset \AA$ and $\R(\S) \subset \BB$, respectively.
Also note that $\|\S\|_{\L(\AA,\BB)} = \sigma_1$ and $\rank(\S)=\sup\{k:\sigma_k>0\}$. 
Moreover, by the Courant-Fisher min-max principle \cite{GolubVanLoan2013}, also known as the Eckart-Young-Mirsky theorem, the $k$th singular value can be characterized by 
\begin{align} \label{eq:minmax}
\sigma_{k}=\min_{\AA_{k-1}}  \max_{a\in \AA_{k-1}^\perp} \|\S a\|_{\BB}/\|a\|_\AA,
%=\max_{\AA_k}\min_{{a\in\AA_k} \|\S a\|_\BB/ \|a\|_\AA},
\end{align}
where $\AA_{k-1}$ are the $(k-1)$-dimensional subspaces of $\AA$.
Hence every linear compact operator 
$\S: \AA \to \BB$ can be approximated by truncated singular value decompositions
\begin{align} \label{eq:TSVD}
\S_Ka = \sum\nolimits_{k \le K} (a, a_k)_\AA \sigma_k b_k, 
\end{align}
with error $\|\S - \S_K\|_{\L(\AA,\BB)} = \sigma_{K+1}$ and the truncated singular value decomposition can be used to construct $\delta$-approximations of minimal rank. 
%
% Further note that any linear bounded operator that can be approximated in norm by finite-rank operators.
%

We further denote by $\HS(\AA,\BB) \subset \L(\AA,\BB)$ the Hilbert-Schmidt class of compact linear operators whose singular values are square summable.
Note that $\HS(\AA,\BB)$ is a Hilbert space equipped with the scalar product $(\S,\R)_{\HS(\AA,\BB)} = \sum_{k\geq 1} (\S a_k,\R a_k)_{\BB}$, where $\{a_k\}_{k\geq 1}$ is an orthonormal basis of $\AA$. Moreover, the scalar product and the associated norm are independent of the choice of this basis. 
Let us mention the following elementary results, which will be used several times later on.
\begin{lemma} \label{lem:hs}
(a) Let $\S \in \HS(\AA,\BB)$. Then there exists a sequence $\{\S_K\}_{K \in \NN}$ of linear operators of rank $K$, such that $\|\S - \S_K\|_{\L(\AA,\BB)} \lesssim K^{-1/2}$. 

(b) Let $\S : \AA \to \BB$, $\R : \BB \to \CC$ be two linear bounded operators and at least one of them Hilbert-Schmidt. Then the composition $\R \mycirc \S : \AA \to \CC$ is Hilbert-Schmidt and
\begin{align*} 
\|\R \mycirc \S\|_{\HS(\AA,\CC)} &\le  \|\R\|_{\L(\BB,\CC)} \|\S\|_{\HS(\AA,\BB)}, \qquad \text{or} \\
\|\R \mycirc \S\|_{\HS(\AA,\CC)} &\le  \|\R\|_{\HS(\BB,\CC)} \|\S\|_{\L(\AA,\BB)}.
\end{align*}
Here and below, we use $a \lesssim b$ to express $a \le C b$ with some generic constant $C$ that is independent of the relevant context, and we write $a \simeq b$ when $a \lesssim b$ and $b \lesssim a$. 
\end{lemma}
For convenience of the reader, we provide a short proof of these assertions.
\begin{proof}
The assumption $\S\in\HS(\AA,\BB)$ implies that $\S$ is compact with square summable singular values, and hence $\sigma_{K,\S}\lesssim K^{-1/2}$. 
The truncated singular value decomposition $\S_K$ then satisfies $\|\S - \S_K\|_{\L(\AA,\BB)} \le \sigma_{K,\S} \lesssim K^{-1/2}$ which yields (a).
After choosing an orthonormal basis $\{a_k\}_{k\geq 1}\subset\AA$, we can write
\begin{align*}
\|\R \mycirc \S\|_{\HS(\AA,\CC)}^2
&=
\sum\nolimits_k\|\R \mycirc \S a_k\|_{\CC}^2 \\
&\leq \|\R\|_{\L(\BB,\CC)}^2 \sum\nolimits_{k\geq 1}\|\S a_k\|_{\BB}^2
= \|\R\|_{\L(\BB,\CC)}^2 \|\S\|_{\HS(\AA,\BB)}^2
\end{align*}
which implies the first inequality of assertion (b). 
The second inequality follows from the same arguments applied to the adjoint $(\R \mycirc \S)^\star = \S^\star \mycirc \R^\star$ and noting that the respective norms of an operator and its adjoint are the same. 
\end{proof}

\subsection{Preliminaries and basic assumptions}

We now introduce in more detail the functional analytic setting for the inverse problem \cref{eq:ip} used for our considerations. 
We assume that the operators $\U$, $\V$, $\D$ appearing in definition \cref{eq:def_T} satisfy  
\begin{assumption}\label{ass:1}
Let $\U \in \HS(\YY,\UU)$, $\V \in \HS(\ZZ,\VV)$, and $\D \in \L(\XX,\L(\UU,\VV'))$.
\end{assumption}
Following our convention, all function spaces appearing in these conditions, except the space $\L(\cdot,\cdot)$, are separable Hilbert spaces. 
We can now prove the following assertions.
\begin{lemma} \label{lem:T}
Let \cref{ass:1} be valid. 
Then $\T(c)= \V' \mycirc \D(c) \mycirc \U$ defines a bounded linear operator  $\T:\XX \to \HS(\YY,\ZZ')$ and, additionally, $\T$ is compact.
\end{lemma}
\begin{proof}
Linearity of $\T$ is clear by construction and the linearity of $\U$, $\V$, and $\D$. Now let $\{y_k\}_{k \ge 1}$ denote an orthonormal basis of $\YY$ and let $c \in \XX$ be arbitrary. 
Then 
\begin{align*}
\|\T(c)\|_{\HS(\YY,\ZZ')}^2 
&= \sum\nolimits_{k\geq 1} \|\V' \mycirc \D(c) \mycirc \U y_k \|^2_{\ZZ'} 
 \le \|\V' \mycirc \D(c)\|_{\L(\UU,\ZZ')}^2  \sum\nolimits_{k\geq 1} \|\U \, y_k\|^2_{\UU} \\
& \le \|\V'\|_{\L(\VV',\ZZ')}^2 \|\D\|_{\L(\XX \to \L(\UU,\VV'))}^2 \|c\|_\XX^2 \, \|\U\|_{\HS(\YY,\UU)}^2,
\end{align*}
where we used \cref{lem:hs} in the second step, 
and the boundedness of the operators in the third.
Since $\|\V'\|_{\L(\VV',\ZZ')} = \|\V\|_{\L(\ZZ,\VV)} \le \|\V\|_{\HS(\ZZ,\VV)}$, we obtain 
\begin{align*}
\|\T(c)\|_{\HS(\YY,\ZZ')} \le \|\U\|_{\HS(\YY,\UU)} \|\V\|_{\HS(\ZZ,\VV)} \|\D\|_{\L(\XX,\L(\UU,\VV'))} \|c\|_\XX    
\end{align*} 
for all $c \in \XX$, which shows that $\T$ is bounded.
Using \cref{lem:hs}(a), we can further approximate $\U$ and $\V$ by operators $\U_K$, $\V_K$ of rank $K$, such that 
\begin{align}\label{eq:UV_decay}
\|\U-\U_K\|_{\L(\YY,\UU)} \lesssim K^{-1/2}
\qquad \text{and} \qquad
\|\V-\V_K\|_{\L(\ZZ,\VV)} \lesssim K^{-1/2},
\end{align}
and we can define an operator $\T_{K,K}: \XX \to \HS(\YY,\ZZ')$ by $\T_{K,K}(c) = \V_K' \mycirc \D(c) \mycirc \U_K$,
which defines an approximation of $\T$ of rank $K^2$. 
From \cref{lem:hs}(b), we infer that 
\begin{align*}
&\|\T-\T_{K,K}\|_{\L(\XX, \HS(\YY,\ZZ'))}
 = \sup_{\|c\|_\XX = 1}\|\V' \mycirc \D(c) \mycirc \U- \V_K'  \D(c) \, \U_K\|_{\HS(\YY,\ZZ')}\\
&\leq (\|\V'-\V_K'\|_{\L(\VV',\ZZ')} \|\U\|_{\HS(\YY,\UU)}
+\|\V'\|_{\HS(\VV',\ZZ')} \|\U-\U_K\|_{\L(\YY,\UU)}) \|\D\|_{\L(\XX,\L(\UU,\VV'))}.
\end{align*}
Using \cref{ass:1} and the bounds \cref{eq:UV_decay}, we thus conclude that $\T$ can be approximated uniformly by finite-rank operators, and hence  $\T$ is compact. 
\end{proof}

\subsection{Sparse tensor product approximation}

As  direct consequence of the arguments used in the previous result, we obtain the following preliminary approximation result.
\begin{lemma}\label{lem:tpapprox}
Let \cref{ass:1} hold. 
Then for any  $\delta>0$ there exists $K\in \NN$ with $K\lesssim\delta^{-2}$ and rank $K$ approximations $\U_K = \U \, \Q_{K,\U} $ and $\V_K = \V \,\Q_{K,\V}$ 
such that
\begin{align}
\|\U-\U_K\|_{\L(\YY,\UU)} \le \delta 
\quad\text{and}\quad
\|\V-\V_K\|_{\L(\ZZ,\VV)} \le \delta.
\end{align}
Here, $\Q_{K,\U}$ and $\Q_{K,\V}$ are orthogonal projections on $\YY$ and $\ZZ$, respectively. 
Furthermore, the operator $\T_{K,K}$ defined by $\T_{K,K}(c) = \V_K' \, \D(c) \, \U_K$ has rank $K^2$ and satisfies
\begin{align}\label{eq:approx_N2}
\|\T-\T_{K,K}\|_{\L(\XX,\HS(\YY,\ZZ'))}\lesssim \delta.
\end{align}
If the singular values of $\U$ and $\V$ satisfy $\sigma_{k,\U},\sigma_{k,\V}\lesssim k^{-\alpha}$ for some $\alpha > 1/2$, then the assertions hold with $K\simeq\delta^{-1/\alpha}$, and consequently $\rank(T_{K,K}) \lesssim \delta^{-2/\alpha}$.
\end{lemma}

Note that different ranks $K_\U$, $K_\V$ could be chosen for the approximations of $\U$ and $\V$, but for ease of notation, we assume $K_\U=K_\V=K$.
Since $\U$ and $\V$ are Hilbert-Schmidt, we know that $\sigma_{k,\U},\sigma_{k,\V}\lesssim k^{-\alpha}$ with $\alpha \ge 1/2$, 
and, thus, the decay assumption on the singular values are not very restrictive.
% , and as a consequence of \eqref{eq:approx_N2}, the singular values of $\T$ decay at least like $\sigma_{k,\T} \lesssim k^{-\beta}$ with $\beta \ge \alpha/2  \ge 1/4$. 
%
% \begin{remark} \rm
% The operators $\U_K$ and $\V_K$ can be obtained by truncated singular value decomposition of $\U$ and $\V$, and $\Q_{K,\U}$ and $\Q_{K,\V}$ then are the projections onto the spaces spanned by the first $K$ right singular vectors of $\U$ and $\V$, respectively. 
% %
% The assertions of the lemma further imply in particular that the singular values of $\T$ decay at least like $\sigma_{k,\T} \lesssim k^{-\beta}$ with $\beta \ge \alpha/2 \ge 1/4$; the latter follows from the fact the $\U$ and $\V$ are Hilbert-Schmidt, and thus their singular values are square summable.  
% \end{remark}
%
%\begin{remark}
%In accordance with \Cref{sec:tpcompr} we set $K=K_\U=K_\V$ to simplify notation. The statements straightforwardly extend to $K_\U\neq K_\V$ and also our numerical experiments are implemented as such.
%\end{remark}
%
In general, $\rank{\T_{K,K}}=K^2$ may however be substantially larger than the optimal rank $N^\text{svd}$ of the truncated singular value decomposition satisfying a similar perturbation bound. 
We now show that based on the approximations $\U_K$, $\V_K$, and assuming sufficient decay of the singular values $\sigma_{k,\U}$, $\sigma_{k,\V}$, one can construct a $\delta$-approximation $\T_{\widehat K}$ for $\T$ of rank $\widehat K \ll K^2$. 
% %\subsection*{Notation}
% \subsection*{Hyperbolic cross approximation}
% Any operator $\S_\delta:\AA \to \BB$ for $\S:\AA \to \BB$ with $\|\S - \S_\delta\|_{\L(\AA,\BB)} \lesssim \delta$ will be called a \emph{$\delta$-approximation for $\S$} in the following.
% %
% Note that $\T_{K,K} = \Q_{K,K} \T$ is a $\delta$-approximation of rank $K^2$, while the $\delta$-approximation of minimal rank is obtained by truncated singular value decomposition \cref{eq:TSVD}. 
% In particular, this implies that
% $\rank(\T_N) \lesssim \rank\T_{K,K}$. 
% %
% We will illustrate now, that the converse statement is in general not true, i.e., the tensor product approximation $\T_{K,K}$ may have substantially higher rank than required for the $\delta$-approximation property.
%
\begin{lemma} \label{lem:approx}
Let $\sigma_{k,\U} \lesssim k^{-\beta}$ and  $\sigma_{k,\V} \lesssim k^{-\alpha}$ (or $\sigma_{k,\U} \lesssim k^{-\alpha}$ and  $\sigma_{k,\V} \lesssim k^{-\beta}$) for some $\beta>1/2$ and $\alpha>\beta+1/2$.
Then $\sigma_{k,\T}\lesssim k^{-\beta}$,
and for any $\delta>0$, we can find $\widehat K \in \NN$ with $\widehat K \lesssim \delta^{-1/\beta}$ and an approximation $\T_{\widehat{K}} = \Q_{\widehat{K}} \T$ of rank $\widehat{K}$, such that 
\begin{align}
\|\T - \T_{\widehat K}\|_{\L(\XX,\HS(\YY,\ZZ'))} \lesssim \delta.
\end{align}
\end{lemma}
\begin{proof}
Let $\{\sigma_{k,*},a_{k,*}, b_{k,*}\}$ denote the singular systems for $\U$ and $\V'$, respectively.
We now show that the hyperbolic cross approximation \cite{Dung2018} 	
\begin{align*}%\label{eq:def_PN}
 \T_{\widehat K}(c) 
 &= \sum\nolimits_{k\geq 1} \sum\nolimits_{\ell=1}^{L_k} \sigma_{\ell,\U} \, \sigma_{k,\V'} \,  (\cdot,a_{\ell,\U})_{\YY} \, \langle \D(c) b_{\ell,\U},a_{k,\V'}\rangle_{\VV'\times\VV} \, b_{k,\V'},
\end{align*}
with the choice $L_k=\lfloor \widehat{K}/k^{1+\eps} \rfloor$, $\widehat{K} \simeq \delta^{-1/\beta}$, and $\eps=(\alpha-\beta-1/2)/(2\beta)>0$ has the required properties.
By counting, one can verify that $\rank(\T_{\widehat K}) \lesssim \sum_{k\geq 1} L_k \lesssim \widehat{K}$, since by construction $L_k \simeq \widehat{K}/k^{1+\eps}$ is summable. 
Furthermore, we can bound
\begin{align*}
\|\T(c) -\T_{\widehat{K}}(c)\|_{\HS(\UU,\VV')}^2 
&= \sum\nolimits_{m\geq 1} \| (\T(c) - \T_{\widehat{K}}(c)) a_{m,\U} \|^2_{\VV'} \\
&= \sum\nolimits_{k\geq 1} \sigma_{k,\V'}^2 \left| \sum\nolimits_{\ell\geq L_k+1} \sigma_{\ell,\U} \langle \D(c) b_{\ell,\U}, a_{k,\V'}\rangle_{\VV' \times \VV} \right|^2 \\
&\le \sum\nolimits_{k\geq 1} \sigma_{k,\V'}^2 \sigma_{L_k}^2 \|\D(c)\|^2_{\L(\UU,\VV')} \|a_{k,\V'}\|_{\VV'}^2.  
\end{align*}
By observing that $\|a_{k,\V'}\|_{\VV'}=1$, $\|\D(c)\|_{\L(\UU,\VV')} \lesssim \|c\|_\XX$, and $\sigma_{k,\V'}=\sigma_{k,\V}$ and by using the decay properties of the singular values, we obtain 
\begin{align*}
\|\T-\T_{\widehat{K}}\|_{\L(\XX,\HS(\UU,\VV'))}^2
&\lesssim \sum\nolimits_{k\geq 1} k^{-2\alpha+2\beta(1+\eps)} \widehat{K}^{-2\beta} 
\lesssim \delta^2.
\end{align*}
In the last step, we used the fact that $-2\alpha+2\beta(1+\eps)<-1$ and $\widehat{K} \simeq \delta^{-1/\beta}$, which follows immediately from the construction.
\end{proof}
\begin{remark} \rm
Comparing the results of \cref{lem:tpapprox,lem:approx}, we expect to obtain a tensor product approximation $\T_{K,K}$ of rank $K^2 \simeq \delta^{-2/\alpha}$ while the hyperbolic cross approximation $\T_{\widehat{K}}$ and consequently also the truncated singular value decomposition of the same accuracy only have rank $\widehat{K} \lesssim \delta^{-1/(\alpha-1/2-\eps)}$, with $\eps=\alpha-\beta-1/2>0$, which may be substantially smaller for $\alpha>1$. 
Note that, like $\T_{K,K}$, the sparse tensor product approximation $\T_{\widehat K}$ can be constructed directly from the low-rank approximations $\U_K$ and $\V_K$.
\end{remark}

\subsection{Quasi-optimal low-rank approximation}\label{sec:recompression}
Let $P_{N^{\text{svd}}}$ denote the orthogonal projection onto the space spanned by the left singular vectors of $\T$ corresponding to the first $N^{\text{svd}}$ singular values, and let $N^\text{svd}$ be chosen such that 
\begin{align*}
    \|P_{N^{\text{svd}}}\T - \T\|_{\L(\XX,\HS(\YY,\ZZ'))} \leq \delta.
\end{align*}
We now show that a compression of any $\delta$-approximation $\T^\delta$, e.g. the tensor product approximation $\T_{K,K}$ or its hyperbolic-cross approximation $\T_{\widehat{K}}$, allows to construct another $\delta$-approximation $\T^\delta_{N^\delta} = P_{N^\delta}^\delta\T^\delta$ for $\T$ with quasi-optimal rank $N^\delta\leq N^{\textrm{svd}}$. 
\begin{lemma}\label{lem:finiterecompression}
Let $\delta>0$ and let $\T^\delta:\XX\to \HS(\YY,\ZZ')$ be a linear compact operator such that $\|\T^\delta-\T\|_{\L(\XX,\HS(\YY,\ZZ'))}\leq C\delta$ for some $C>0$. Let $\P_{N^{\delta}}^{\delta}\T^\delta$ denote the truncated singular value decomposition of $\T^\delta$ with minimal rank $N^{\delta}$ such that
\begin{align}\label{eq:assumption_finite_recompression}
    \|\T^\delta-\P_{N^{\delta}}^{\delta} \T^\delta\|_{\L(\XX,\HS(\YY,\ZZ'))}\leq (C+1)\delta.
\end{align}
Then $N^{\delta}\leq N^{\textrm{svd}}$ and
\begin{align*}%\label{eq:assertion_finite_recompression}
    \|\T-\P_{N^{\delta}}^{\delta} \T^\delta\|_{\L(\XX,\HS(\YY,\ZZ'))}\leq (2C+1)\delta,
\end{align*}
i.e., $\P_{N^{\delta}}^{\delta} \T^\delta$ 
is a $\delta$-approximation for $\T$ with quasi-optimal rank. 
\end{lemma}
\begin{proof}
\emph{Step 1.}
We start by recalling a well-known perturbation result for singular values \cite{Kato1966}, i.e., we show that for each $k\in\NN$ one has 
\begin{align}\label{eq:pertub_sing}
\sigma_{k}-C\delta \leq \sigma_k^\delta \leq \sigma_{k}+C\delta,
\end{align}
where $\{\sigma_k\}$ and  $\{\sigma_k^\delta\}$ denote the singular values of $\T$ and $\T^\delta$, respectively. Let us abbreviate $\| \cdot \| = \|\cdot \|_{\L(\XX,\HS(\YY,\ZZ'))}$, choose $\varepsilon>0$, and let $\P_{M}^{\text{svd}}\T$ denote the truncated singular value decomposition of $\T$ with optimal rank such that 
$
\|\P_{M}^{\text{svd}}\T-\T\| < \varepsilon.
$
Using the optimality of $M$ and the non-expansiveness of the projection, we estimate
\begin{align*}
    \| (\Id-\P_{M}^{\delta})\T^\delta\|&\leq \|(\Id-\P_{M}^{\text{svd}})\T^\delta\|\\
    &\leq \|(\Id-\P_{M}^{\text{svd}})\T\|+\|(\Id-\P_{M}^{\text{svd}})(\T-\T^\delta)\|
    \leq \varepsilon + C\delta.
\end{align*}
For $\varepsilon=\sigma_{k+1}$, we have $M=M(\varepsilon)=k$, and we conclude that  $\| (\Id-\P_{M}^{\delta})\T^\delta\|=\sigma_{M+1}^\delta$ and $\sigma_{k+1}^\delta\leq \sigma_{k+1}+C\delta$. 
The second inequality in \eqref{eq:pertub_sing} follows by considering $\T$ as a $\delta$-approximation for the operator $\T^\delta$.

\emph{Step 2.} 
Now let $N^{\delta}$ be as in the statement of the lemma and $N^{\text{svd}}=M(\delta)$ as defined in Step~1. 
Then \eqref{eq:pertub_sing} implies that $\sigma_{N^{\text{svd}}+1}^\delta \leq \sigma_{N^{\text{svd}}+1}+ C\delta \leq (C+1)\delta$. Optimality of $N^{\delta}$ 
%and countability of the spectrum imply
then implies $\sigma_{N^{\text{svd}}+1}^\delta\leq\sigma_{N^{\delta}+1}^\delta\leq(C+1)\delta<\sigma_{N^{\delta}}^\delta$, and from the monotonicity of the singular values, we conclude that $N^{\delta}\leq N^{\text{svd}}$. Furthermore,
$$
\|\P_{N^{\delta}}^{\delta} \T^\delta - \T \|\leq \|\P_{N^{\delta}}^{\delta} \T^\delta-\T^\delta\|+\|\T^\delta-\T\|\leq (2C+1)\delta,
$$
i.e., $\P_{N^{\delta}}^{\delta}\T^\delta$ is a $\delta$-approximation for $\T$ with quasi-optimal rank $N^{\delta}\leq N^{\text{svd}}$.
\end{proof}

We now apply the previous lemma with $\T^\delta=\T_{\widehat{K}}$ and $\P_N^\delta=\P_N$ the projection of the  corresponding truncated singular value decomposition, which leads to
\begin{align}
    \T_N = \P_N \T_{\widehat{K}} =  \P_N (\Q_{\widehat{K}}\T).
\end{align}
By construction this is a $\delta$-approximation for $\T$ of quasi-optimal rank $N \le N^{\text{svd}}$.

\subsection{Summary}

Let us briefly summarize the main observations and results of this section. 
Based on 
$\delta$-approximations $\U_K$, $\V_K$ for the operators $\U$ and $\V$,  
we can construct a low-rank $\delta$-approximation $\T_{\widehat{K}}=\Q_{\widehat{K}} \T$ for $\T$ via hyperbolic-cross approximation. 
We observed that in typical situations $\T_{\widehat{K}}$ is of much lower rank than the corresponding full tensor product approximations $\T_{K,K}$,
whose assembly can be completely avoided.
By truncated singular value decomposition of $\T_{\widehat{K}}$, 
we obtained another $\delta$-approximation 
$\T_N = \P_N \T_{\widehat{K}} =  \P_N (\Q_{\widehat{K}}\T)$ with quasi-optimal rank $N \le N^\text{svd}$. 
In summary, we thus efficiently computed a low-rank approximation $\T_N$ for $\T$ with similar rank and approximation properties as the truncated singular value decomposition. 

The analysis in this section was done in abstract function spaces and applies verbatim to infinite-dimensional operators as well as to their finite-dimensional (truth) approximations obtained after discretization. As a consequence, the computational results, e.g., the ranks $K$ and $N$ of the approximations, can be  expected to be essentially independent of the actual truth approximation used for computations.
In the language of model-reduction, the low-rank approximation $\T_N$ is a certified reduced order model.

\section{Fluorescence optical tomography} \label{sec:3}

In order to illustrate the viability of the theoretical results derived in the previous section, we now consider in some detail a typical application arising in medical imaging. 

\subsection{Model equations}

Fluorescence optical tomography aims at retrieving information about the concentration $c$ of a fluorophore inside an object by illuminating this object from outside with near infrared light and measuring the light reemitted by the fluorophores at a different wavelength. 
The distribution $u_x=u_x(q_x)$ of the light intensity inside the object generated by a source $q_x$ at the boundary is described by 
\begin{alignat}{5}
 -\nabla\cdot(\kappa_x \nabla u_x) + \mu_x u_x &= 0, \qquad &&\text{in } \Omega, \label{eq:ux1}\\
 \kappa_x\partial_nu_x + \rho_x u_x &= q_x, \qquad && \text{on } \partial\Omega. \label{eq:ux2}
\end{alignat}
We assume that $\Omega\subset\RR^d$, $d=2,3$, is a bounded domain with smooth boundary enclosing the object under consideration.
The light intensity $u_m=u_m(u_x,c)$ emitted by the fluorophores is described by a similar equation
\begin{alignat}{5}
  -\nabla\cdot(\kappa_m \nabla u_m) + \mu_mu_m &= c u_x, \qquad && \text{in } \Omega, \label{eq:um1}\\
  \kappa_m\partial_nu_m + \rho_m u_m &=0, \qquad &&\text{on } \partial\Omega. \label{eq:um2}
\end{alignat}
The model parameters $\kappa_i$, $\mu_i$, and $\rho_i$, $i=x,m$, characterize the optical properties of the medium at excitation and emission wavelength; we assume these parameters to be known, e.g., determined by independent measurements \cite{ArridgeSchotland2009}. 
As shown in \cite{EggFreiSch10}, the above linear model, which can be interpreted as a Born approximation or linearization, is a valid approximation for moderate fluorophore concentrations.

\subsection{Forward operator}

The forward problem in fluorescence optical tomography models an experiment in which the emitted light resulting from excitation with a known source and after interaction with a given fluorophore concentration is measured at the boundary. 
The measurable quantity is the outward photon flux, which is proportional to $u_m$; see \cite{ArridgeSchotland2009} for details. 
The potential data for a single excitation with source $q_x$ measured by a detector with characteristic $q_m$ can be described by 
\begin{align}\label{eq:forward_fdot}
    \big\langle {\T (c)} \, q_x, q_m \big\rangle = \int_{\partial\Omega} u_m q_m \, ds(x),
\end{align}
where $u_m$ and $u_x$ are determined by the boundary value problems \cref{eq:ux1}--\cref{eq:um2}.
The inverse problem finally consists of determining the concentration $c$ of the fluorophore marker from measurements $\langle \T(c) q_x,q_m\rangle$ for multiple excitations $q_x$ and detectors $q_m$.

We now illustrate that fluorescence optical tomography perfectly fits into the abstract setting of \Cref{sec:2}. 
Let us begin with defining the excitation operator 
\begin{align}\label{eq:op_U}
    \U: H^{1}(\partial\Omega)\to H^{1}(\Omega), \qquad q_x\mapsto \U q_x := u_x,
\end{align}
which maps a source $q_x$ to the corresponding weak solution $u_x$ of \cref{eq:ux1}--\cref{eq:ux2}.
The interaction with the fluorophore can be described by the multiplication operator
\begin{align}\label{eq:multiplication_op}
    \D: L^2(\Omega)\to \mathcal{L}(H^{1}(\Omega),H^{1}(\Omega)'), \qquad \D(c) u = c u. 
\end{align}
In dimension $d \le 3$, the product $c u$ of two functions $c \in L^2(\Omega)$ and $u \in H^1(\Omega)$, lies in $L^{3/2}(\Omega)$ and can thus be interpreted as a bounded linear functional on $H^1(\Omega)$; this shows that $\D$ is a bounded linear operator. 
We further introduce the linear operator 
\begin{align}\label{eq:op_V}
    \V:H^{1}(\partial\Omega) \to H^{1}(\Omega), \qquad q_m\mapsto \V q_m:=v_m,
\end{align}
which maps $q_m$ to the weak solution $v_m$ of the adjoint emission problem
\begin{alignat}{5}
-\nabla\cdot(\kappa_m \nabla v_m) + \mu_m v_m &=  0, \qquad &&\text{in } \Omega, \label{eq:umadj1}\\
\kappa_m\partial_n v_m + \rho_m v_m &= q_m, \qquad &&\text{on } \partial\Omega.\label{eq:umadj2}
\end{alignat}
One can verify that $\V$ is the dual of the solution operator $u_{m\mid\partial\Omega} = \V' \D(c) u_x$ of the system \cref{eq:um1}--\cref{eq:um2}; see \cite{EggFreiSch10} for details. 
Hence we may express the forward operator as 
\begin{align}\label{eq:forward_fdot_TVDcU}
    {\T (c)} = \V' \mycirc \D(c) \mycirc \U.
\end{align}
As function spaces we choose $\UU=\VV=H^1(\Omega)$, $\YY=\ZZ=H^1(\partial\Omega)$, and $\XX=L^2(\Omega)$. 

In order to apply the results of \Cref{sec:2}, it remains to verify \cref{ass:1}. 
We already showed that $\D \in \L(\XX,\L(\UU,\VV'))$ is a bounded linear operator. The following assertion states that also the remaining conditions on $\U$ and $\V$ hold true.
\begin{lemma}\label{lem:applicability}
The operators $\U$ and $\V$ defined in \cref{eq:op_U} and \cref{eq:op_V} are Hilbert-Schmidt and their singular values decay like $\sigma_{k,\U}\lesssim k^{-3/(2d-2)}$ and $\sigma_{k,\V}\lesssim k^{-3/(2d-2)}$.
\end{lemma}
\begin{proof}
The Hilbert-Schmidt property follows immediately from the decay behavior of the singular values. 
%
%To show the latter, 
Let $\Th$ be a quasi-uniform triangulation of the domain $\Omega$ of meshsize $h$ and $\partial\Th$ be the induced segmentation of the boundary $\partial\Omega$. Further, let $\YY_h = P_1(\partial\Th) \cap H^1(\partial\Omega) \subset H^{-1/2}(\partial\Omega)$ be the space of piecewise linear finite elements on $\partial\Th$. Let $\Q_h$ be the $L^2$-orthogonal projection onto $\YY_h$ and $q\in H^1(\partial\Omega)$ arbitrary. Then standard approximation error estimates, see e.g., \cite{BS2008}, yield
\begin{align*}
\|q - \Q_h q\|_{H^{-1/2}(\partial\Omega)} 
\lesssim h^{3/2} \|q\|_{H^1(\partial\Omega)}.  
\end{align*}
A-priori estimates for elliptic PDEs
yield $\|\U q\|_{H^1(\Omega)} \lesssim \|q\|_{H^{-1/2}(\partial\Omega)}$, and hence $\U$ can be continuously extended to an operator on $H^{-1/2}(\partial\Omega)$; see e.g. \cite{Evans}. 
This yields
\begin{align*}
\|\U- \U \Q_h\|_{\L(H^1(\partial\Omega),H^1(\Omega))} \lesssim h^{3/2} \lesssim k^{-3/(2d-2)},
\end{align*}
where $k=\dim(\YY_h) = \rank(\Q_h) \simeq h^{-(d-1)}$ is the dimension of the space $\YY_h$.     
From the min-max characterization of the singular values \cref{eq:minmax}, we may therefore conclude that $\sigma_{k,\U} \lesssim k^{-3/(2d-2)}$ as required. 
The result for $\sigma_{k,\V}$ follows in the same way.
\end{proof}

\begin{remark} \rm
If prior knowledge $\operatorname{supp}(c) \subset \Omega$ on the support of the fluorophore concentration is available, which is frequently encountered in practice, 
elliptic regularity \cite{Evans} implies exponential decay of the singular values $\sigma_{k,\U}$ and $\sigma_{k,\V}$.
In such a situation, the ranks $K$ and $N$ in \cref{lem:tpapprox,lem:finiterecompression} will depend only logarithmically on the noise level $\delta$, and an accurate approximation $\T_N$ of very low rank can be found. 
\end{remark}

{
\begin{remark}\rm \label{rem:unique}
If $c_1$ and $c_2$ are two fluorophore concentrations leading to the same measurements, that is $\T(c_1)=\T(c_2)$, then the factorization \cref{eq:forward_fdot_TVDcU} shows that
\begin{align}
    \int_\Omega (c_1-c_2) u_x v_m \,dx =0
\end{align}
for all possible excitation fields $u_x$ satisfying \cref{eq:ux1} and adjoint emission fields $v_m$ satisfying \cref{eq:umadj1}. Under certain regularity conditions, density results for the set of products $\{u_x v_m\}$ imply $c_1=c_2$, see \cite[Chapter~5]{Isakov2017} for precise statements. Hence, in such a situation, the solution $c^\dagger$ of \cref{eq:ip} is unique.
\end{remark}}

\section{Algorithmic realization and complexity estimates} \label{sec:4}

We will now discuss in detail the implementation of the model reduction approach presented in \Cref{sec:2} for the fluorescence optical tomography problem and demonstrate its viability by some preliminary considerations.
For ease of presentation, we consider a simple two-dimensional test problem. Our observations, however, carry over almost verbatim also to three dimensional problems of similar dimensions.

\subsection{Problem setup}
For the discretization of \cref{eq:ux1}--\cref{eq:ux2} and \cref{eq:umadj1}--\cref{eq:umadj2}, we use a standard finite element method with continuous piecewise linear polynomials. 
The computational meshes used for the truth approximations are obtained by successive uniform refinement of the initial mesh, leading to quasi-uniform conforming triangulations $\Th$ of the domain $\Omega$ with $h>0$ denoting the mesh size. Thus, the corresponding spaces $\UU_h,\VV_h\subset H^1(\Omega)$ then have dimension $m\simeq h^{-d}$ each. 
For our test problem, we have $d=2$, since we consider a two-dimensional setting.
We choose the same finite element space $\XX_h$ also for the approximation of the concentration $c$. 
The sources $q_x,q_m$ for the forward and the adjoint problem are approximated by piecewise linear functions on the boundary of the same mesh $\Th$; hence $\YY_h$, $\ZZ_h \subset H^1(\partial\Omega)$ have dimension $k\simeq h^{d-1}$.
All approximation spaces are equipped with the topologies induced by their infinite dimensional counterparts.
Standard error estimates allow to quantify the discretization errors in the resulting truth approximation of the forward operator and to establish the $\delta$-approximation property for $h$ small enough.
The error introduced by the discretization can therefore be assumed to be negligible.

A sketch of the domain $\Omega \subset \RR^2$ and the coarsest mesh $\Omega_h$ used for our computations as well as the parameter $c^\dag$ to be identified are depicted in \cref{fig:setup}.
\begin{figure}[ht!]
\centering
\includegraphics[width=0.25\textwidth]{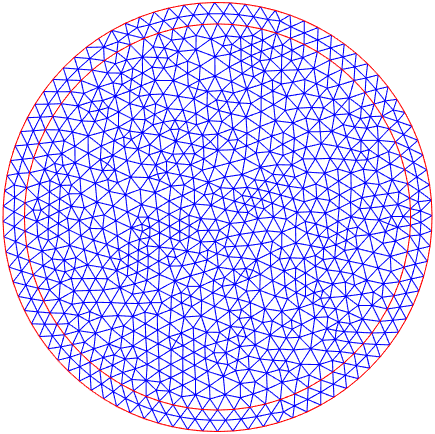}
\hspace*{3em}
\includegraphics[width=0.25\textwidth]{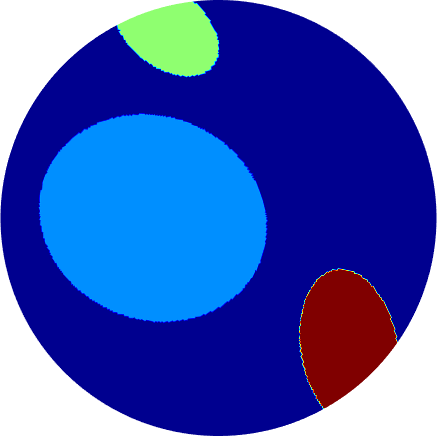}
\hspace*{3em}
\includegraphics[width=0.25\textwidth]{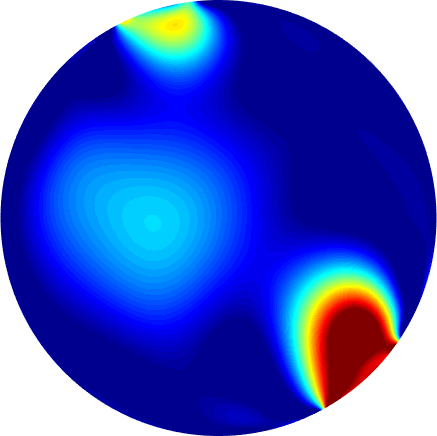}
\caption{Left: computational domain and coarsest mesh used for our computations. Middle: minimum norm solution $c^\dag$. Right: reconstructed fluorophore concentration $c_\alpha^\delta$ for $\delta=10^{-5}$. \label{fig:setup}}
\end{figure}
The characteristic dimension of the relevant function spaces after discretization can be deduced from \cref{tab:dimension}. The numbers in the table also illustrate the assumption $k < m < k^2$ and that the discretized inverse problem is formally overdetermined.
\begin{table}[ht!]
\centering \footnotesize
\setlength{\tabcolsep}{8pt} % Default value: 6pt
\renewcommand{\arraystretch}{1.2} % Default value: 1
\caption{Dimensions $m=\dim(\XX_h)=\dim(\UU_h)=\dim(\VV_h)$ and $k=\dim(\YY_h)=\dim(\ZZ_h)$ of the relevant function spaces 
and discretization errors $de_h=\|\T_h - \T\|_{\L(\XX,\HS(\YY,\ZZ'))}$ between approximation $\T_h$ on refinement level $\text{ref}$ and the truth approximation $\T$ on level $\text{ref}=5$. The norm of the forward operator is $\|\T_h\|_{\L(\XX,\HS(\YY,\ZZ'))}=0.2804$ on all mesh levels.
\label{tab:dimension}}
\begin{tabular}{c||r|r|r|r|r|r}
    ref      &  0 & 1 & 2 & 3 & 4 & 5 \\
    \hline
    \hline
    $m$      &  993 & 3\,881 & 15\,345 & 61\,025 & 243\,393 & 927\,161 \\
    \hline
    $k$      & 88 & 176 & 352 & 704 & 1\,408 & 2\,816 \\
    \hline
    \hline
    $de_h$ & $6.31 \cdot 10^{-4}$ & $1.69\cdot 10^{-4}$ & $4.34\cdot10^{-5}$ & $1.10\cdot10^{-5}$ & $2.75\cdot10^{-6}$ & ---
\end{tabular}
\end{table}
%
%\herbert{Norm des Operators $\|\T\| \approx ***$}
Note that an approximation on level $\textrm{ref}\ge4$ is required to guarantee a discretization error of $de_h =\|\T_h - \T\|_{\L(\XX;\HS(\YY,\ZZ'))} \le 10^{-5}$.
Further observe that for the finest mesh level $\textrm{ref}=5$, the discretized forward operator amounts to a linear mapping $\ttT$ from $\RR^{927\,161}$ to $\RR^{2\,816 \times 2\,816}$.  The storage of the matrix $\ttA$ representing the forward operator would require approximately 56TB of memory and even one single evaluation of $\ttT(\ttc)$  via the matrix product $\ttA \ttc$ would require approximately 7Tflops.  
It should be clear that more sophisticated algorithms are required to make even the evaluation of the forward operator feasible.

\subsection{Truth approximation} \label{sec:41}
Let us briefly discuss in a bit more detail the algebraic structure of the resulting problems arising in the truth approximation.
Under the considered setup, the finite element approximation of problem \cref{eq:ux1}--\cref{eq:ux2} leads to the linear system 
\begin{align}\label{eq:fem_ux}
(\ttK_\ttx+\ttM_\ttx+\ttR_\ttx) \, \ttU &= \ttE_\ttx\ttQ_\ttx.
\end{align}
Here $\ttK_\ttx, \ttM_\ttx \in \RR^{m\times m}$ are the stiffness and mass matrices with coefficients $\kappa_x$, $\mu_x$, and the matrices $\ttR_\ttx\in\RR^{m\times m}$, $\ttE_\ttx \in \RR^{m \times k}$ stem from the discretization of the boundary conditions.
The columns of regular $\ttQ_\ttx \in \RR^{k \times k}$ represent the individual independent sources in the basis of $\YY_h$. 
Any excitation generated by a source in $\YY_h$ can thus be expressed as a linear combination of columns of the excitation matrix $\ttU \in \RR^{m\times k}$, which serves as a discrete counterpart of the operator $\U$.
In a similar manner, the discretization of the adjoint problem \cref{eq:umadj1}--\cref{eq:umadj2} leads to  \begin{align}\label{eq:fem_vm}
(\ttK_\ttm + \ttM_\ttm + \ttR_\ttm) \, \ttV &= \ttE_\ttm \ttQ_\ttm.
\end{align}
whose solution matrix $\ttV\in\RR^{m\times k}$ can be interpreted as the discrete counterpart of the operator $\V$.
The system matrices $\ttK_\ttm$, $\ttM_\ttm$, $\ttR_\ttm$, and $\ttE_\ttm$ have a similar meaning as above, and the columns of $\ttQ_\ttm$ represent the individual detector characteristics. Recall from \Cref{sec:reginv} that $k=k_\YY=k_\ZZ$ only to simplify exposition.
The algebraic form of the truth approximation finally reads
\begin{align} \label{eq:truth}
\ttT(\ttc) = \ttV^\top \, \ttD(\ttc) \, \ttU,
\end{align}
where $\ttD(\ttc)\in\RR^{m\times m}$ is the matrix representation of the finite element approximation for the operator $\D(c_h)$ with $\ttc \in \RR^m$ denoting the coordinates of the function $c_h\in\XX_h$.
The discrete measurement $\ttM_{ij} = (\ttV^\top \ttD(\ttc) \ttU)_{ij}=\ttV(:,i)^\top \ttD(\ttc) \ttU(:,j)$ then approximates the data taken by the $i$th detector for excitation with the $j$th source.

From the particular form \cref{eq:truth} of the forward operator, 
one can see that only the matrices $\ttU$, $\ttV$, and a routine for the application of $\ttD(\ttc)$ are required to evaluate $\ttT(\ttc)$. 
In our example, the application of $\ttD(\ttc)$ amounts to the multiplication by a diagonal matrix{, which leads to a complexity of $O(k^2m)$ flops for the application and a memory requirement of $O(2km+m)$ bytes for the forward operator.}
Note that the matrices $\ttU$ and $\ttV$ can now be stored on a standard workstation, while the application of the forward operator is still too compute intensive to be useful for the efficient solution of the inverse problem under consideration. 
\begin{remark} \rm \label{rem:orthornormal}
Let $\ttS\ttY,\ttS\ttZ$ be the matrix representation of the $H^1(\partial\Omega)$ inner products for the spaces $\YY_h$, $\ZZ_h$. Furthermore, let $\ttA\ttY, \ttA\ttZ \in \RR^{k \times k}$ be orthogonal with respect to $\ttS\ttY$ and $\ttS\ttZ$, i.e., $\ttA\ttY^\top * \ttS\ttY * \ttA\ttY = \ttI$ and $\ttA\ttZ^\top * \ttS\ttZ * \ttA\ttZ = \ttI$. 
Then $\ttA\ttY = \ttQ\ttx *\ttA\ttx$ and $\ttA\ttZ=\ttQ\ttm * \ttA\ttm$ with $\ttA\ttx=\ttQ\ttx^{-1} * \ttA\ttY$ and $\ttA\ttm=\ttQ\ttm^{-1} * \ttA\ttZ$, and the Hilbert-Schmidt norm of the measurement matrix $\ttM=\ttT(\ttc)$ can be expressed by the Frobenius norm
\begin{align*}
\|\ttM\|_{\HS} := \|\ttA\ttm^\top * \ttM * \ttA\ttx \|_{\fro}.     
\end{align*}
Note that we simply have $\|\ttM\|_{\HS} := \|\ttM\|_\fro$ if the columns of $\ttQ\ttx$, $\ttQ\ttm$ are chosen orthonormal with respect to the $\ttS\ttY$ and $\ttS\ttZ$ inner products right from the beginning. 
We will use this fact in our numerical tests below.
\end{remark}
\begin{remark}
Using multigrid solvers, the matrices $\ttU$ and $\ttV$ can be computed in  $O(mk)$ operations \cite{FreibergerEtAl11}, which is, at least asymptotically, negligible compared to the application of $\ttT(\ttc)$ in tensor product form.
In our computational tests, we utilize sparse direct solvers for the computation of $\ttU$ and $\ttV$, for which the computational cost is $O(m^{3/2} + k m \log(m))$. 
Since $m^{3/2} \le m k$ and $\log(m) \le k$ in our two-dimensional setting, this is still of lower complexity than even a single evaluation of $\ttT(\ttc)$.
\end{remark}

\subsection{Model reduction -- offline phase}\label{sec:num.offline}
With $\ttU$ and $\ttV$ obtained, we are now in the position to compute our reduced order model.

\subsubsection{Orthonormalization} \label{sec:422}
Let  \verb+SX,SY,SZ+ be the Gram matrices representing the scalar products of the function spaces $\XX_h$, $\YY_h$, $\ZZ_h$. 
As a next step, we compute the approximations for the singular value decompositions of the excitation and emission operators. For this, we recall that the right singular vectors of an operator $\U$ correspond to the eigenvectors of $\U^\star \U$. The singular value decompositions for the matrices $\ttU$ and $\ttV$ can thus be computed by the generalized eigenvalue decompositions 
\smallskip
\begin{verbatim}
        [Ax,Dx]=eigs(U'*SX*U,SY);  
        [Am,Dm]=eigs(V'*SX*V,SZ);  
\end{verbatim}
\smallskip
Note that some slight modifications would be required here, if the source and detector matrices \verb+Qx+ and \verb+Qm+ would not be chosen as the identity matrices. 
The columns of \verb+Ax+ and \verb+Am+ are orthogonal with respect to the \verb+SY+ and \verb+SZ+ scalar product and thus define bases of the discrete source and detector spaces. After appropriate scaling, the columns can be assumed to be normalized such that \verb+Ax'*SY*Ax+ and \verb+Am'*SZ*Am+ equal the identity matrix.
To simplify the subsequent discussion, we change the definition of the sources and detectors as well as of the excitation and emission matrices, and redefine the forward operator according to
\smallskip
\begin{verbatim}
        Qx=Qx*Ax;  U=U*Ax;
        Qm=Qm*Am;  V=V*Am;  
        T=@(c) V'*D(c)*U;
\end{verbatim}
\smallskip
The columns of \verb+Qx+ and \verb+Qm+ are now orthogonal with respect to the \verb+SY+ and \verb+SZ+ scalar products, and as a consequence, the Hilbert-Schmidt norm in the measurement space amounts to the Frobenius norm of \verb+M=T(c)+; see \cref{rem:orthornormal} for details.

{The memory cost for storing $\ttA\ttY=\ttU'*\ttS\ttX*\ttU$ and $\ttA\ttZ=\ttV'*\ttS\ttX*\ttV$ amounts to $O(k^2)$ bytes, while the computation of the matrix products requires $O(k^2 m)$ flops. 
The complexity for the eigenvalue decompositions finally is $O(k^3)$ flops.
Note that the setup of the matrices $\ttA\ttY$ and $\ttA\ttZ$ has the same cost as a single evaluation $\ttT(\ttc)$ of the forward operator.
The additional memory required for storing the $k \times k$ matrices $\ttA\ttY$ and $\ttA\ttZ$ is negligible.} 

\subsubsection{Low-rank approximations for $\ttU$ and $\ttV$} \label{sec:423}

The eigenvalues computed in the decompositions above correspond to the square of the singular values of $\ttU$ and $\ttV$. 
We here allow for different ranks in the approximation and define
truncation indices
\smallskip
\begin{verbatim}
        dx=diag(Dx); xKK=find(dx>delta^2); xK=length(xKK); 
        dm=diag(Dm); mKK=find(dm>delta^2); mK=length(mKK);
\end{verbatim}
\smallskip
We could further set \verb+K=max(xK,mK)+ 
to stay exactly with the notation used in \Cref{sec:2}, but we stress that our implementation is in full generality.
The low-rank approximations for $\ttU$ and $\ttV$ and the resulting tensor product approximation of the forward operator $\ttT(\ttc)$ are then given by 
\smallskip
\begin{verbatim}
        QxK=Qx(:,xKK);  UK=U(:,xKK);
        QmK=Qm(:,mKK);  VK=V(:,mKK);
        TKK=@(c) VK'*D(c)*UK;
\end{verbatim}
\smallskip
Observe that the measurements \verb+MKK=TKK(c)+ obtained by this approximation correspond to a sub-block of the full measurements, i.e., \verb+MKK=M(mKK,xKK)+.  

\subsubsection{Hyperbolic cross approximation} \label{sec:424}

The proof of \cref{lem:approx} shows that we may replace the tensor product operator \verb+TKK(c)+ by the hyperbolic cross approximation \verb+TK(c)+, which takes into account only the entries \verb+M(k,l)=MKK(k,l)+ of the measurements \verb+M=T(c)+ for indices $\verb+k+\cdot\verb+l+ \le N \lesssim \delta^{-\beta}$. 
In our computations, we actually replace $N$ by $K$, i.e., we utilize the hyperbolic cross approximation $\verb+TK(c)+$ of \verb+TKK(c)+. 
The assembly of the matrix representation $\ttA\ttK$ for $\ttT\ttK(\ttc)=\ttA\ttK*\ttc$ then reads
%\smallskip 
\begin{verbatim}
        m=0;
        for k=1:K
            for l=1:floor(K/k)
                m=m+1;
                AK(m,:)=(VK(:,k)'.*UK(:,l)')*DD;
            end
        end
\end{verbatim}
\smallskip 
Here, \verb+DD+ is a diagonal matrix representing the numerical integration on the computational domain.
Let us note that $\ttA\ttK$ and thus also the operator \verb+TK(c)+ do not have a tensor product structure any more; therefore the measurements \verb+MK=TK(c)+ are stored as a column vector rather than a matrix. 
The norm in the reduced measurement space then is the Euclidean norm for vectors.
Also note that the construction of $\ttA\ttK$ and \verb+TK+ only requires access to the matrices \verb+UK+ and \verb+VK+ defining the operator \verb+TKK+.

The tensor product approximation \verb+TKK(c)=VK'*D(c)*UK+ requires only subblocks \verb+UK, VK+ of the excitation and emission matrices \verb+U, V+ and, therefore, no additional memory cost arises in setting up this approximation. 
To achieve a $\delta$-approximation with $\delta=10^{-3}$, for instance, we expect to require approximately $\ttK=100$ singular components of \verb+U+ and \verb+V+; see \cref{lem:applicability} for details. The tensor product approximation will then have rank $\ttK^2=10^4$. 
For the hyperbolic cross approximation \verb+TK+, we however expect to require only approximately $2\ttK=200$ components of the tensor product approximation \verb+TKK+; compare with \cref{lem:approx}. 
In \cref{tab:approx} we summarize the expected memory and computational cost for the corresponding approximations. 
\begin{table}[ht!]
\centering \footnotesize
\setlength{\tabcolsep}{8pt} % Default value: 6pt
\renewcommand{\arraystretch}{1.2} % Default value: 1
\caption{Memory and computation cost for storing and applying the tensor product approximation $\ttT\ttK\ttK$ of rank $\ttK^2=10^4$ and the corresponding hyperbolic cross approximation $\ttT\ttK$ of rank $2\ttK=200$.
The theoretical memory and computation cost is given by \emph{mem($\ttT\ttK\ttK$)}$=$\emph{mem($\ttT\ttK$)}$=16\ttK m$ \emph{bytes}, \emph{ops($\ttT\ttK\ttK(\ttc)$)}$=\ttK m+\ttK^2m$ \emph{flops}, and 
\emph{ops($\ttT\ttK(\ttc)$)}$=2\ttK m$ 
% \matthias{for rank $\ttK$ ok, but isn't it rather $\ttK\log\ttK$?}
\emph{flops}, respectively.\label{tab:approx}}
\begin{tabular}{c||c|c|c|c|c|c}
    ref      &  0 & 1 & 2 & 3 & 4 & 5 \\
    \hline
    \hline
    mem (GB)        & 0.001 & 0.006 & 0.023 & 0.009 & 0.036 & 1.448 \\
    \hline
    ops ($\ttT\ttK\ttK(\ttc)$, Gflop)  & 0.009 & 0.037 & 0.144 & 0.574 & 2.289 & 9.144 \\
    \hline 
    ops ($\ttT\ttK(\ttc)$, Gflop)   & 0.000 & 0.001 & 0.003 & 0.011 & 0.045 & 0.181 \\
\end{tabular}
\end{table}
Note that the tensor product structure allows to store the tensor product approximation \verb+TKK+ as efficiently as the hyperbolic cross approximation \verb+TK+. The application of the latter is, however, substantially more efficient. 

\subsubsection{Final recompression} \label{sec:425}

The last step in our model reduction approach consists in a further compression of the hyperbolic cross approximation \verb+TK+ of the tensor product operator \verb+TKK+; cf. \Cref{sec:recompression} for details. This can be realized by 
\smallskip
\begin{verbatim}
        AKt=DD\AK';
        AKAKt=AK*AKt;
        [VA,DA]=eigs(AKAKt,NK);
\end{verbatim}
\smallskip 
where \verb+NK=size(AK,1)+ is the number of terms used for the hyperbolic cross approximation. The recompression then consists of selecting the largest entries, i.e., 
\smallskip 
\begin{verbatim}
        N=find(diag(DA)>delta^2);
        AN=AK*VA(:,N); ANt=DD\AN';
\end{verbatim}
\smallskip     
Matrix representations for the projection operators $\Q_{K,K}$, $\P_K$, and $\P_N$ corresponding to the tensor product, the hyperbolic cross, and the final approximation, can be assembled easily from the eigenvectors computed during the construction. 

As before, the compression is based on singular value decompositions of operators via the solution of generalized eigenvalue problems for the matrices 
\verb+BKK=AKK*(DD\AKK')+ respectively \verb+BK=AK*(DD\AK')+, 
where \verb+AKK+ and \verb+AK+ are the matrix representations for the operators \verb+TKK(c)+ and \verb+TK(c)+ respectively. 
The computational cost for the assembly of \verb+BKK+ and \verb+BK+ is listed in \cref{tab:final}.
For an evaluation of the computational complexity, we again assume that \verb+TKK+ has rank $\ttK^2=10^4$ and that \verb+TK+ is of rank $2\ttK=200$.
\begin{table}[ht!]
\centering \footnotesize
\setlength{\tabcolsep}{8pt} % Default value: 6pt
\renewcommand{\arraystretch}{1.2} % Default value: 1
\caption{Complexity for computing 
\emph{$\ttB\ttK\ttK=\ttA\ttK\ttK{\rm *}(\ttD\ttD\backslash\ttA\ttK\ttK'$)} and \emph{$\ttB\ttK=\ttA\ttK{\rm *}(\ttD\ttD\backslash\ttA\ttK'$)}.
The estimates are 
\emph{mem($\ttB\ttK\ttK$)}$=8\ttK^4$ \emph{bytes}, \emph{ops($\ttB\ttK\ttK$)=}$m\ttK^4$ \emph{flops} 
and \emph{mem($\ttB\ttK$)=}$32\ttK^2$ \emph{bytes}, \emph{ops($\ttB\ttK$)=}$2\ttK$ \emph{flops}. 
\label{tab:final}}
\begin{tabular}{c||l|l|l|l|l|l}
    ref                &  0 & 1 & 2 & 3 & 4 & 5 \\
    \hline
    \hline
    ops ($\ttB\ttK\ttK$, Gflop)   & 92.480 & 361.44 & 1\,429.1 & 5\,683.4 & 22\,667 & 90\,540 \\
    \hline
    ops ($\ttB\ttK$, Gflop)    & 0.000 & 0.001 & 0.003 & 0.011 & 0.045 & 0.181 \\
\end{tabular}
\end{table}
Let us note that the required memory for storing \verb+BKK+ and \verb+BK+ is independent of the mesh size; for the setting considered here, it is given by \verb+mem(BKK)+$=745$MB and \verb+mem(BK)+$=0.3$MB.
Assuming that an eigenvalue decomposition of an $n \times n$ matrix needs roughly \verb+ops(eig)+$=50n^3$ operations, we obtain \verb+ops(eig(BKK))+$=46\,566$ \verb+Gflops+ and \verb+ops(eig(BK))+$=0.373$ \verb+Gflops+. 
Even if a computationally more efficient low-rank approximation \cite{HMT2011,Stoll2012} for the tensor product operator \verb+TKK+ would be used, the evaluation of \verb+TKK(c)+ remains rather expensive; see \cref{tab:approx} for details.
Therefore, the tensor product approximation \verb+TKK+ is not really useful for the computation of low-rank approximation on large computational meshes.
As shown in \Cref{sec:2}, a quasi-optimal approximation \verb+TN+ can be computed also by truncation of the singular value decomposition of the hyperbolic cross approximation \verb+TK+, which does not require any additional computations.

\subsection{Online phase}

After the construction of the low-rank approximation \verb+TN(c)+ as outlined above, the actual solution of the inverse problem consists of three basic steps; see \Cref{sec:1} for a brief explanation. 
The first step is the \emph{data compression} which can be expressed as {\texttt{MN=PN*vec(MKK)} with $\ttM\ttK\ttK=(\ttQ\ttm\ttK'{\rm *}\ttM){\rm *}\ttQ\ttx\ttK$}. Here we make explicit use of the tensor product structure, which allows us to efficiently compress the data already during recording. After this, only the second projection \verb+PN+ has to be applied. The additional memory required for computing {$\ttM\ttK\ttK$ is $O(k \ttK)$ bytes for each, $\ttQ\ttx\ttK$, $\ttQ\ttm\ttK$, and $\ttQ\ttm\ttK'{\rm *}\ttM$ and thus negligible.}
Note that storing the full data \verb+M+ requires {$O(k^2)$ bytes} %\verb+mem(M)+$=k^2$ 
which is substantially higher.
{The computational cost of the data compression step is $O(\ttK^2 k+k^2 \ttK)$.}
As mentioned before, the data can be partially compressed already during recording, such that access to the full data is actually never required. 

The final compression \verb+MN=PN*MKK+ is independent of the system dimensions $m,k$ and its computational cost is therefore negligible. 
The same applies for the solution of the regularized inverse problem \verb~zadN=(ANANt+alpha*I)\MN~, which is the second step in the online phase and only depends on the dimension \verb+N+ of the reduced model. 

The synthesis of the solution according to \cref{eq:cand} can finally be realized by simple multiplication \verb+cadN=ANt*zadN+, where \verb+ANt+ denotes the matrix representation of the adjoint of the fully reduced forward operator \verb+TN+. {The additional memory required for $\ttA\ttN\ttt$ is $O(m\ttN)$ bytes, whereas the computation of $\ttc\tta\ttd\ttN$ can be accomplished in $O(m\ttN)$ flops.}
Thus, as claimed in the introduction, the most compute intensive part of the online phase is the data compression, even if the tensor product structure is utilized to compress the data already during recording.

\section{Computational results} \label{sec:45}

We now illustrate the practical performance of our model reduction approach for the test problem introduced in the previous section.
For comparison, we also report on corresponding results for 
traditional iterative methods for solving the inverse problem \cref{eq:ip}--\eqref{eq:def_T}, as well as for methods based on a tensor-product approximation. 
A snapshot of the geometry, the exact solution, and a typical reconstruction is depicted in \cref{fig:setup}.

For our numerical tests, the model parameters are set to 
$\kappa_x=1$, $\mu_x=0.2$, $\rho_x=10$ and 
$\kappa_m=2$, $\mu_m=0.1$, and $\rho_m=10$.
We further assume prior knowledge that $c$ is supported in a circle of radius $0.9$, i.e., the distance of its support to the boundary $\partial\Omega$ is at least $0.1$. 
The singular values of the operators $\U$ and $\V$ as well as of $\T$ can thus be assumed to decay exponentially.
The noise level in \cref{eq:noise} is set to $\delta=10^{-5}$ and the regularization parameter $\alpha$ is chosen from $\{10^{-n}\}$ via the discrepancy principle. In all our computations, this led to $\alpha=10^{-8}$ which complies to the theoretical prediction for exponentially ill-posed problems \cite{EHN96}. 

Let us note that, in order to ensure sufficient accuracy $\|\T_h - \T\|_{\L(\XX;\HS(\YY,\ZZ'))} \le \delta$ of the truth approximation, one should utilize a discretization of the forward operator on mesh level $\text{ref} \ge 4$ for the reconstruction; see Table~\ref{tab:dimension} and Section~\ref{sec:1}.  
For evaluation of computational performance, we however also report about results on coarser meshes.

All computations are performed on on a workstation with Intel(R) Xeon(R) Gold 6130 CPU @ 2.10GHz and 768GB of memory. In our tests we use only a single core of the processor and an implementation in \textsc{Matlab} 9.6.0.

\subsection{Problem initialization}

This step consists of setting up the excitation and emission matrices $\ttU$, $\ttV$, which are required for the efficient evaluation of $\ttT(\ttc)$. 
In \cref{tab:init}, we also report about the singular value decomposition of $\ttU$ and $\ttV$, by which we orthogonalize the sources and detectors; as mentioned in \Cref{sec:num.offline}, this is required for computation of the Hilbert-Schmidt norm $\|\M^\delta\|_{\HS(\YY,\ZZ')}$ of the measurement operator.
\begin{table}[ht!]
\centering \footnotesize
\setlength{\tabcolsep}{8pt} % Default value: 6pt
\renewcommand{\arraystretch}{1.2} % Default value: 1
\caption{Computation times (sec) for the individual steps in the problem setup phase.\label{tab:init}}
\begin{tabular}{c||l|l|l|l|l|l}
    ref
        &  0 & 1 & 2 & 3 & 4 & 5 \\
    \hline
    \hline
    initialization of $\ttU,\ttV$  
        & 0.01 & 0.06 & 0.34 & 3.45 & 47.74 & 612.63 \\
    \hline
    setup of $\ttU'*\ttD\ttX*\ttU$, $\ttV'*\ttD\ttX*\ttV$       
        & 0.00 & 0.01 & 0.16 & 2.24 & 30.68 & 460.13 \\
    \hline
    eigenvalue decompositions 
        & 0.01 & 0.03 & 0.24 & 2.15 & 33.44 & 470.88 \\
    \hline
    orthogonalization of $\ttU$, $\ttV$ 
        & 0.00 & 0.02 & 0.09 & 1.70 & 25.12 & 363.64 \\
\end{tabular}
\end{table}
While the theoretical complexity of the first and third step is somewhat smaller than that of the second and fourth step, the overall computation times for the individual steps in the setup phase are comparable. 
Note that the computations reported in Table~\ref{tab:init} are required for the solution of the inverse problem \eqref{eq:ip}, independent of the particular solution strategy, and they can be performed in a pre-processing step

\subsection{Model reduction -- offline phase}

The singular values computed in the decompositions of $\ttU$ and $\ttV$ allow to determine the truncation indices $\ttx\ttK$ and $\ttm\ttK$ used to define the $\delta$-approximations \verb+UK=U(:,xKK)+ and \verb+VK=V(:,mKK)+. 
The values of $\ttx\ttK$ and $\ttm\ttK$ obtained in our numerical tests are depicted in \cref{tab:truncation}.
\begin{table}[ht!]
\centering \footnotesize
\setlength{\tabcolsep}{8pt} % Default value: 6pt
\renewcommand{\arraystretch}{1.2} % Default value: 1
\caption{Truncation indices $\ttx\ttK$ and $\ttm\ttK$ guaranteeing $\|\ttU-\ttU\ttK\|\le \delta$ and $\|\ttV-\ttV\ttK\| \le \delta$ with $\delta=10^{-5}$. \label{tab:truncation}}
\begin{tabular}{c||l|l|l|l|l|l}
    ref
        &  0 & 1 & 2 & 3 & 4 & 5 \\
    \hline
    \hline
    $\ttx\ttK$ 
        & 88 & 139 & 163 & 179 & 185 & 187 \\
    \hline
    $\ttm\ttK$
        & 88 & 121 & 133 & 137 & 137 & 137 \\
\end{tabular}
\end{table}
On the coarsest mesh, the number of possible excitations and detectors is limited by the number of boundary vertices, but otherwise, the number of truncation indices $\ttx\ttK$ and $\ttm\ttK$ are almost independent of the truth approximation. This can be expected since the eigenvalues converge with increasing refinement of the mesh. 

\subsubsection{Forward evaluation}

As outlined in \Cref{sec:num.offline}, the full operator and its tensor product approximation can now be simply defined by \verb+T=@(c) V'*D(c)*U+ and \verb+TKK=@(c) VK'*D(c)*UK+. 
In \cref{tab:evaluation}, we report about the computation times for a single evaluation of these operators.
\begin{table}[ht!]
\centering \footnotesize
\setlength{\tabcolsep}{8pt} % Default value: 6pt
\renewcommand{\arraystretch}{1.2} % Default value: 1
\caption{Computation times (sec) for a single evaluation of $\ttT(\ttc)$ and $\ttT\ttK\ttK(\ttc)$. \label{tab:evaluation}}
\begin{tabular}{c||l|l|l|l|l|l}
    ref
        &  0 & 1 & 2 & 3 & 4 & 5 \\
    \hline
    \hline
    $\ttT(\ttc)$ 
        & 0.00 & 0.01 & 0.07 & 1.03 & 13.99 & 214.20 \\
    \hline
    $\ttT\ttK\ttK(\ttc)$
        & 0.00 & 0.01 & 0.02 & 0.11 & 0.52 & 2.20 \\
\end{tabular}
\end{table}
As can be seen, even the problem adapted evaluation of the full operator $\ttT(\ttc)$ becomes practically useless for the solution of the inverse problem \cref{eq:ip}. The tensor product approximation $\ttT\ttK\ttK(\ttc)$, which is the underlying approximation for methods based on optimal sources \cite{Krebs2009} or based on the Kathri-Rhao product \cite{Markel2019}, seems somewhat better suited but, as we will see below, may still be not appropriate for the efficient solution of the inverse problem. 

\subsubsection{Truncated singular value decomposition}

As a theoretical reference for model-reduction, we consider the low-rank approximation of $\ttT(\ttc)$ and $\ttT\ttK\ttK(\ttc)$ by truncated singular value decomposition, which can be computed via eigenvalue decompositions for the symmetric operators $\ttT (\ttTt(\ttM))$ and $\ttT\ttK\ttK(\ttT\ttK\ttKt(\ttM\ttK\ttK))$. 
The latter can be computed numerically by the \texttt{eigs} routine of \textsc{Matlab} in a matrix-free way, i.e., only requiring the application of the operators $\ttT(\ttc)$, $\ttT\ttK\ttK(\ttc)$ and their adjoints $\ttTt(\ttM)$, $\ttT\ttK\ttKt(\ttM)$. 
The sum of \verb+xK+ and \verb+mK+ specifies the maximal number of eigenvalues to be considered by the algorithm.
In \cref{tab:eigs}, we display the computation times for eigenvalue solvers and the number $N$ of relevant eigenvalues required to obtain a $\delta$-approximation. 
\begin{table}[ht!]
\centering \footnotesize
\setlength{\tabcolsep}{8pt} % Default value: 6pt
\renewcommand{\arraystretch}{1.2} % Default value: 1
\caption{Computation times (sec)  for singular value decompositions of $\ttT(\ttc)$ and $\ttT\ttK\ttK(\ttc)$ and truncation indices $N$ leading to corresponding $\delta$-approximations. \label{tab:eigs}}
\begin{tabular}{c||l|l|l|l|l|l}
    ref
        &  0 & 1 & 2 & 3 & 4 & 5 \\
    \hline
    \hline
    svd($\ttT$) 
        & 6.46 & 28.23 & 284.33 & --- & --- & --- \\
    \hline
    $N(\ttT)$ 
        & 231 & 303 & 473 & --- & --- & --- \\
    \hline
    \hline
    svd($\ttT\ttK\ttK$)
        & 6.45 & 15.05 & 48.40 & 248.42 & 994.66 & --- \\
    \hline
    $N(\ttT\ttK\ttK)$ 
        & 231 & 276 & 296 & 310 & 314 & --- \\
\end{tabular}
\end{table}
The computation times for the decomposition of the full operator $\ttT$ increase roughly by a factor of $8$ per refinement, while those for 
the tensor product approximation only increase by a factor of $4$. Computations taking longer than 1000sec were not conducted. 
Due to the substantially smaller rank, the evaluation $\ttT\ttN(\ttc)$ of the low-rank approximations resulting from one of the 
singular value decompositions above is faster by a factor of more than 100 compared to that of the tensor product approximation $\ttT\ttK\ttK(\ttc)$,
and even on the finest mesh only takes about 0.01sec. 
Let us recall that a discretization at mesh level $\text{ref}\ge4$ is required to guarantee sufficient approximation $\|\T_h - \T\|_{\L(\XX;\HS(\YY,\ZZ'))} \le \delta$ of the truth approximation $\T_h$ used for the solution of the inverse problem.

\subsubsection{Setup of reduced order model}

As described in Section~\ref{sec:2}, we can utilize the hyperbolic cross approximation $\ttT\ttK$ instead of the full tensor product approximation $\ttT\ttK\ttK$ without loosing the $\delta$-approximation property. 
In \cref{tab:hyperbolic}, we summarize the computation times for assembling the hyperbolic cross approximation $\ttT\ttK$
and the subsequent singular value decomposition used in the final  recompression step. 
\begin{table}[ht!]
\centering \footnotesize
\setlength{\tabcolsep}{8pt} % Default value: 6pt
\renewcommand{\arraystretch}{1.2} % Default value: 1
\caption{Computation times (sec) for construction of the hyperbolic cross approximation $\ttT\ttK(\ttc)$ and its singular value decomposition 
used for constructing the final approximation $\ttT\ttN(\ttc)$ with rank $N(\ttT\ttK)$. \label{tab:hyperbolic}}
\begin{tabular}{c||l|l|l|l|l|l}
    ref
        &  0 & 1 & 2 & 3 & 4 & 5 \\
    \hline
    \hline
    setup of $\ttT\ttK$, $\ttT\ttK\ttT\ttK\ttt$ 
%         & 0.01 & 0.05 & 0.44 & 1.91 & 11.80 & 57.88 \\  %% AK
%         & 0.00 & 0.08 & 0.94 & 3.39 & 18.68 & 82.44 \\  %% AKAKt
        & 0.01 & 0.013 & 1.38 & 6.31 & 30.48 & 140.32 \\    %% AK+AKAKt
    \hline
    svd($\ttT\ttK$) 
%         & 0.04 & 0.61 & 3.86 & 4.71 & 4.95 & 4.48 \\  %% eig()
        & 0.08 & 0.87 & 2.57 & 3.51 & 3.87 & 4.03 \\ %% eigs()
    \hline
    \hline
    $\text{rank}(\ttT\ttK)$ 
        & 403 & 933 & 1\,725 & 1\,867 & 1\,905 & 1\,917 \\
    \hline
    $\text{rank}(\ttT\ttN)$ 
        & 166 & 266 & 391 & 396 & 401 & 403 \\
\end{tabular}
\end{table}
Note that the setup cost for the hyperbolic cross approximation increases roughly by a factor of $4$ for each refinement, 
while the subsequent singular value decomposition and the ranks are essentially independent of the mesh level.

Due to the moderate rank $K=\operatorname{rank}(\ttT\ttK)$ of the hyperbolic cross approximation, 
it pays off to compute the matrix approximation of $\ttT\ttK\ttT\ttKt$ and to use it for the 
subsequent eigenvalue decomposition.
As can be seen from \cref{tab:hyperbolic}, the recompression step allows to reduce the rank by another factor of about $5$. 
As predicted by our theoretical investigations, the rank of the final approximation $\ttT\ttN$ is comparable 
to that of the truncated singular value decomposition of the full operator $\ttT$ or its tensor product 
approximation $\ttT\ttK\ttK$; cf. \cref{tab:eigs}.
The use of the hyperbolic cross approximation $\ttT\ttK$ instead of the full operator or its tensor product approximation however allows to speed up the computation of the final low-rank approximation $\ttT\ttN$ substantially.
Again, the rank of the approximation becomes essentially independent of the mesh after some initial refinements, reflecting the mesh-independence of our approach.

\subsection{Solution of inverse problem -- online phase}

We now turn to the online phase of the solution process.
Iterative methods are used for the solution of the inverse problem
with the full operator $\ttT$ and its tensor product approximation $\ttT\ttK\ttK$. 
As mentioned before, we choose a regularization parameter  $\alpha=10^{-8}$, which was determined by the discrepancy principle. 
For the computation of the regularized solution \eqref{eq:cand} with full operator $\ttT(\ttc)$ and the tensor product approximation $\ttT\ttK\ttK(\ttc)$, we use \textsc{Matlab}'s \verb+pcg+ routine with tolerance set to $\texttt{tol}=\alpha \delta^2$.
Since the rank of the final reduced order model  $\ttT\ttN$ is rather small, we can use a direct solution of \eqref{eq:cand} by \textsc{Matalb}'s \texttt{backslash} operator in that case.

In \cref{tab:iterative}, we display 
the online solution times 
and the error $\texttt{err}=\|c_\alpha^\delta-c^\dag\|$ obtained for the final iterate. 
\begin{table}[ht!]
\centering \footnotesize
\setlength{\tabcolsep}{8pt} % Default value: 6pt
\renewcommand{\arraystretch}{1.2} % Default value: 1
\caption{Computation times (sec) for the solution of the inverse problem via Tikhonov regularization. %\eqref{eq:cad}.
Iterative methods are utilized for the solution of \eqref{eq:cad} in the first two cases working with operators $\ttT$ and $\ttT\ttK\ttK$, while a direct solver is used for the low-rank approximation $\ttT\ttN$. 
\label{tab:iterative}}
\begin{tabular}{c||l|l|l|l|l|l}
    ref
        &  0 & 1 & 2 & 3 & 4 & 5 \\
    \hline
    \hline
    $\ttT$
        & 1.24 & 13.91 & 320.73 & --- & --- & --- \\    
    \hline
    % err($\ttT)$ 
    %     & 0.107 & 0.106 & 0.106 & --- & --- & --- \\
    % \hline
    $\ttT\ttK\ttK$ 
        & 1.22 & 10.07 & 65.02 & 382.76 & --- & --- \\    
    \hline
    \hline
    % err($\ttT\ttK\ttK)$ 
    %     & 0.107 & 0.106 & 0.106 & 0.106 & --- & --- \\
%    \hline
    $\ttT\ttN$ 
        & 0.01 & 0.01 & 0.03 & 0.13 & 0.52 & 1.94 \\   
\end{tabular}
\end{table}
Approximately $1\,800$ iterations are required for the iterative solution of \cref{eq:cad} with the full operator $\ttT$ and the tensor-product approximation $\ttT\ttK\ttK$ on all mesh levels, which again illustrates the mesh-independence of the algorithms. 
Note that even for the tensor product approximation, the iterative solution on fine meshes becomes practically infeasible, while for the low-rank approximation $\T_N$ of quasi-optimal rank, the inverse problem solution remains extremely efficient up the finest mesh. 

In \cref{tab:solution} we discuss in more detail the computation times for the individual steps in \cref{eq:cand}, namely the \emph{data compression},
the \emph{solution of the regularized normal equations}, and the \emph{synthesis of the reconstruction}.
\begin{table}[ht!]
\centering \footnotesize
\setlength{\tabcolsep}{8pt} % Default value: 6pt
\renewcommand{\arraystretch}{1.2} % Default value: 1
\caption{Computation times (sec) for the individual steps of the online phase for inversion with reduced order model $\T_N$ of quasi-optimal rank. \label{tab:solution}}
\begin{tabular}{c||l|l|l|l|l|l}
    ref
        &  0 & 1 & 2 & 3 & 4 & 5 \\
    \hline
    \hline
    data compression 
        & 0.001 & 0.005 & 0.028 & 0.114 & 0.457 & 1.831 \\    
    \hline
    regularized normal equations 
        & 0.002 & 0.001 & 0.002 & 0.003 & 0.003 & 0.003 \\
    \hline
    synthesis 
        & 0.001 & 0.001 & 0.004 & 0.015 & 0.061 & 0.107 \\ 
\end{tabular}
\end{table}
Similar online computation times are also obtained for the low-rank approximation computed by truncated singular value decomposition of the full operator $\T$,
since its rank and approximation properties are very similar to that of the approximation constructed by our approach.

As announced in the introduction and predicted by our complexity estimates,
the data compression step becomes the most compute-intensive task in the 
online solution via the low-rank reduced order model $\ttT\ttN$. 
While the data compression and synthesis step depend on the 
dimension of the truth approximation, the solution of the regularized normal 
equations becomes completely independent of the computational mesh. 
Also observe that the quality of the reconstruction is not degraded by 
the use of a low-rank approximation in the solution process.
Overall, we thus obtained an extremely efficient, stable, and accurate reconstruction for fluorescence tomography. 

\section{Summary} \label{sec:5}

A novel approach towards the systematic construction of approximations for high dimensional linear inverse problems with operator valued data was proposed yielding certified reduced order models of quasi-optimal rank. 
The approach was fully analyzed in a functional analytic setting and the theoretical results were illustrated by an application to fluorescence optical tomography.
The main advantages of our approach, compared to more conventional low-rank approximations, like truncated singular value decomposition, lies in 
a vastly improved setup time and the possibility to partially compress the data already during recording.
In particular, the computational effort of setting up the reduced order model $\T_N$ is comparable to that of \emph{one single} evaluation $\T(c)$ of the forward operator. Due to the underlying tensor-product approximation, access to the full data $\M^\delta$ is not required.

The most compute intensive part in the offline phase consists in the setup of the discrete representations for $\U$ and $\V$ as well as their eigenvalue decomposition. A closer investigation and the use of parallel computation could certainly further improve the computation times for this step. 
Further acceleration of the data compression and synthesis step could probably be achieved by using computer graphics hardware. 
The low-dimensional reduced order models obtained in this paper may also serve as preconditioners for the iterative solution of related nonlinear inverse problems, which  would substantially increase the field of potential applications.

\section*{Acknowledgments}

The work of the second author was supported by the German Research Foundation (DFG) via grants TRR~146 C3 and TRR~154 C04 and via the ``Center for Computational Engineering'' at TU Darmstadt.

\bibliographystyle{siamplain} 
\bibliography{accelfdot}

\end{document}